\documentclass[11pt,a4paper,reqno]{amsart}
\usepackage{cmap}
\usepackage[T1]{fontenc}
\usepackage{graphicx}
\usepackage{subcaption}
\usepackage{setspace,mathrsfs,amsthm,amsmath,amssymb,amsfonts,lipsum,appendix,mathtools,cite,tabularx,fontenc,bbm,caption}

\usepackage[shortlabels]{enumitem}
\makeatletter
\newtheorem*{rep@theorem}{\rep@title}
\newcommand{\newreptheorem}[2]{%
	\newenvironment{rep#1}[1]{%
		\def\rep@title{#2 \ref{##1}}%
		\begin{rep@theorem}}%
		{\end{rep@theorem}}}
\makeatother

\newtheorem{theorem}{Theorem}[section]

\newtheorem{proposition}[theorem]{Proposition}
\theoremstyle{definition}
\newtheorem{remark}[theorem]{Remark}
\newtheorem{definition}[theorem]{Definition}

\usepackage[nointegrals]{wasysym}
\usepackage[misc]{ifsym}
\usepackage[dvipsnames]{xcolor}
\definecolor{ao}{rgb}{0.0, 0.5, 0.0}
\definecolor{lasallegreen}{rgb}{0.0, 0.3, 0.0}
\usepackage{hyperref,tikz}
\hypersetup{
	colorlinks=true,
	linkcolor=blue,
	filecolor=magenta,      
	urlcolor=violet,
	citecolor=OrangeRed,
}
\usepackage[margin=0.75in]{geometry}
\usetikzlibrary{calc,patterns,angles,quotes,intersections}

\usepackage{wrapfig}

\let\oldnorm\norm
\def\norm{\@ifstar{\oldnorm}{\oldnorm*}}

\newcommand{\Om} {\Omega}

\newcommand{\Lom} {\mathcal{L}}

\newcommand{\ol}[1] {\overline{#1}}

\newcommand{\dw}{{\,\rm d}(w)}
\newcommand{\dz}{{\,\rm d}(z)}

\newcommand{\G}{{\mathcal G}}
\newcommand{\disk}{{\mathbb D}}
\newcommand{\s}{{\mathbb{S}}}

\newcommand{\R}{{\mathbb R}}

\newcommand{\ra} {\rightarrow}

\newcommand{\lb} {\langle}
\newcommand{\rb} {\rangle}

\def\H{{\mathcal H}}

\usepackage[hyperpageref]{backref}

\usepackage{orcidlink}

\usepackage{a4wide}
\usepackage{csquotes}

\usepackage{amsfonts}
\usepackage{amssymb,amsthm}
\usepackage{amsmath}
\allowdisplaybreaks
\usepackage{mathtools}
\mathtoolsset{showonlyrefs}

\numberwithin{equation}{section}

\usepackage {setspace}

\usepackage{enumitem}

\setlist{nosep}

\usepackage{todonotes}

\begin{document}
\singlespacing

\title[On Eigenvalues of Logarithmic Potential Operator in the Hyperbolic Space]{On Eigenvalues of Logarithmic Potential Operator in the Hyperbolic Space}

\author[Jiya Rose Johnson]{Jiya Rose Johnson}
\author[Sheela Verma]{Sheela Verma}

\address[Jiya Rose Johnson]{\newline\indent
	Department of Mathematics,
	Indian Institute of Technology Madras 
	\newline\indent
	Chennai 36, India
}
\email{jiyarosejohnson@gmail.com}
\address[Sheela Verma]{\newline\indent
	Department of Mathematical Sciences,                                  Indian Institute of Technology (BHU) Varanasi 
	\newline\indent
	Uttar Pradesh 221005
	\newline\indent
	\orcidlink{0000-0001-6443-7736} 0000-0001-6443-7736 
}
\email{sheela.mat@iitbhu.ac.in}









\begin{abstract}
Let $\Omega$ be a bounded open set in the Poincar\'e hyperbolic disk, $\disk$. In this article, we consider the hyperbolic logarithmic potential operator 
$\mathcal{L}_h : L^2(\Omega) \to L^2(\Omega)$, defined by
\begin{equation*}
    \mathcal{L}_h u(z)=\frac{1}{2}\int_\Omega \log\frac{1}{[z,w]}\,u(w)\, \dw,
\end{equation*}
and the associated eigenvalue problem on $\Om$
\begin{equation}
    \mathcal{L}_h u=\tau u.
\end{equation}
We first extend the notion of polarization with respect to hyperplanes in the Poincar\'e disk and prove the associated properties. Then we establish a reverse Faber-Krahn inequality for the largest eigenvalue, $\tau_{h}$ of $\Lom_h$, under polarization. Further, we provide a representation formula for the eigenfunctions of $\mathcal{L}_h$. In addition, we show that the operator $\mathcal{L}_h$ is a positive operator on $L^2(\Om)$.
\end{abstract}

\maketitle
\definecolor{lblack}{gray}{0.3}
\definecolor{mygray}{gray}{0.9}
\definecolor{vlgray}{gray}{0.96}
\definecolor{medgray}{gray}{0.8}
\definecolor{dgray}{gray}{0.7}
\begin{quote}	
	\setcounter{tocdepth}{1}
	\tableofcontents
	\addtocontents{toc}{\vspace*{0ex}}
\end{quote}
 \section{Introduction}

For a bounded open set $\Omega \subset \mathbb{R}^2$, the logarithmic potential operator 
$\mathcal{L} : L^2(\Omega) \to L^2(\Omega)$ is defined by
\begin{equation}
    \mathcal{L} u(x)
    = \frac{1}{2\pi} \int_\Omega \log \frac{1}{|x-y|}\, u(y)\, dy.
\end{equation}
The operator $\mathcal{L}$ is self-adjoint and compact on $L^2(\Omega)$\cite[p.367]{troutman1967}, and hence admits infinitely many eigenvalues. 
The associated eigenvalue problem
\begin{equation}\label{oldevproblem}
    \mathcal{L} u = \tau u
\end{equation}
has been extensively studied; see 
\cite{anoopjiya2025, troutman1967, troutman1969, suraganlog2016, Anderson1992, Alsenafi_etal_2024, Kac1970}. In \cite{anoopjiya2025}, reverse Faber-Krahn inequalities were established for the largest eigenvalue under polarization and Schwarz symmetrization. The existence of a unique negative eigenvalue for \eqref{oldevproblem} was characterized in \cite{Kac1970, troutman1967} in terms of a geometric quantity known as the transfinite diameter. Moreover, in \cite{troutman1967}, the author gave an explicit representation for eigenfunctions of \eqref{oldevproblem}. The eigenvalues of \eqref{oldevproblem} on the disks were also analyzed in terms of the zeros of the Bessel functions in \cite{Anderson1992, suraganvolume}.  In this paper, we consider analogues operator in the hyperbolic space and explore the above results in the Poincar\'e disk model 
$\mathbb{D} \subset \mathbb{C}$.

We recall the following definitions from \cite[Chapter 2]{Stoll_2016}. The hyperbolic arclength element on the Poincar\'e disk $\mathbb{D}$ is given by
\begin{equation}
    ds = \frac{2\,|dz|}{1 - |z|^2},
\end{equation}
where $z=x+iy$ , $|z|^2=x^2+y^2$, and $|dz|=\sqrt{dx^2+dy^2}$. The hyperbolic length of a $C^{1}$ curve $\gamma : [0,1] \to \mathbb{D}$ is given by
\[
L(\gamma) = \int_{0}^{1} \frac{2\,|\gamma'(t)|}{1 - |\gamma(t)|^{2}}\, dt.
\]
For points $z,w \in \mathbb{D}$, the hyperbolic distance $d_h(z,w)$ is defined by
\begin{equation}
    d_h(z,w) = \inf_{\gamma \in \mathcal{C}_{z,w}} L(\gamma),
\end{equation}
where
\[
\mathcal{C}_{z,w} = \Big\{ \gamma : [0,1] \to \mathbb{D} \,\Big|\, 
\gamma \text{ is } C^1,\ \gamma(0)=z,\ \gamma(1)=w \Big\}.
\]
It is well known that, for any $z,w \in \mathbb{D}$,
\begin{equation}
    d_h(z,w)= 2 \tanh^{-1} \left| \frac{z-w}{1-\overline{z}w} \right|.
\end{equation}
Let $\Omega \subset \mathbb{D}$ be an open set, bounded with respect to the hyperbolic metric $d_h$. Then
\begin{equation}
    L^{2}(\Omega)
    = \left\{ f \text{ measurable on } \Omega : \int_{\Omega} |f(w)|^{2}\, \dw < \infty \right\},
\end{equation}
where
\begin{equation}
   \dw
    = \frac{1}{\pi}\,\frac{dA(w)}{(1-|w|^{2})^{2}},
\end{equation}
and $dA$ denotes the Lebesgue area measure on $\mathbb{D}$. 
We also recall the pseudo-hyperbolic metric,
\begin{equation}
    [z,w] = \left| \frac{z-w}{1-\overline{z}w} \right|.
\end{equation}
Using the pseudo-hyperbolic distance and the above definitions, we define the following logarithmic potential operator in the hyperbolic space $\disk$. 
\begin{definition}\label{def_log_hyper}{\textbf{(The hyperbolic logarithmic potential operator).}}
    Let $\Omega \subset \mathbb{D}$ be an open set, bounded with respect to the hyperbolic metric $d_h$. The hyperbolic logarithmic operator
    $\mathcal{L}_h : L^2(\Omega) \to L^2(\Omega)$ is defined by
    \begin{equation*}
    \mathcal{L}_h (u)(z)
    = \frac{1}{2}\int_{\Omega} \log\frac{1}{[z,w]} u(w)\, \dw,
    \qquad z \in \Omega.
    \end{equation*}
\end{definition}


 Similar to the fundamental solution of the usual Laplacian in $\R^2$, the fundamental solution $g_h$ of the hyperbolic Laplacian $\Delta_h$ on the unit disk $\disk$\cite[Sections 3.1 \& 3.2]{Stoll_2016} is given by 
\begin{equation*}
    g_h(|x|) = \frac{1}{2} \log \frac{1}{|x|}.
\end{equation*}
The solution to the Poisson equation $-\Delta u = f$ in $\mathbb{R}^2$ can be expressed as the convolution of $f$ with $g(x)=\frac{1}{2\pi}\log \tfrac{1}{|x|}$, the fundamental solution of Laplacian in $\R^2$. More precisely,
\begin{equation}
u(x) = (g * f)(x)
= \frac{1}{2\pi} \int_{\mathbb{R}^2} \log \frac{1}{|x-y|}\, f(y)\, dy,
\end{equation}
and $u$ satisfies $-\Delta u = f$ in $\mathbb{R}^2$ in the distributional sense. For a fixed $x \in \mathbb{R}^2$, the translation map $y \mapsto y - x$ carries $\mathbb{R}^2$ onto itself. In analogy with this, the map
\begin{equation}\label{eqn_phi_z_w}
    \phi_z(w) = \frac{z - w}{1 - \overline{z}\, w}, \qquad z \in \mathbb{D},
\end{equation}
maps the unit disk $\mathbb{D}$ to itself. This leads to the definition of the invariant convolution of measurable functions $f$ and $g$ on $\mathbb{D}$\cite[Section 3.4]{Stoll_2016} by
\[
    (f \star g)(z) = \int_{\mathbb{D}} f(w)\, g(\phi_z(w))\, dw,
\]
whenever the integral is well defined. Under suitable assumptions on $f$, the convolution $f \star g_h$ yields a solution of the Poisson equation $-\Delta_h u = f$ in the unit disc; see \cite[Lemma~6.9, p.~67]{Stoll_1994}. For bounded open subsets of $\R^2$, the logarithmic potential $\mathcal{L}u$ is defined as the convolution of $u$ with the fundamental solution $g$. Analogously, on the Poincar\'e disk $\disk$, the hyperbolic logarithmic potential operator is defined by convolving $u$ with the fundamental solution $g_h$ using the invariant convolution mentioned above. This naturally motivates the use of the pseudo-hyperbolic metric rather than the standard hyperbolic metric, since $|\phi_z(w)|=[z,w]$.

\par The authors of \cite{suraganfirstsecond} studied Riesz transforms on Riemannian manifolds. They proved a Rayleigh-Faber-Krahn inequality for the largest eigenvalue and a Hong-Krahn-Szeg\"{o} inequality for the second largest eigenvalue of the Riesz transform.
 The motivation to study the operator $\Lom_h$ is inspired by \cite{suraganfirstsecond} and from the discussion presented in \cite[Chapter 3]{Stoll_2016}. We study the eigenvalue problem 
\begin{equation}\label{evproblem}
    \Lom_h u=\tau u,
\end{equation}
similar to \eqref{oldevproblem}, in the hyperbolic space. We say $(\tau,u)$ is an eigenpair for \eqref{evproblem}, if it satisfies the weak formulation:
\begin{equation}
    \lb \Lom_h u,v \rb=\tau \lb u,v\rb,\quad \forall\, v\in L^2(\Om),
\end{equation}
where $\lb \cdot ,\cdot\rb$ denotes the usual inner product defined in $L^2(\Om)$.  More precisely, $(\tau,u)$ is an eigenpair if it satisfies the equation:
\begin{equation}\label{eulerlagrange}
         \frac{1}{2} \iint\limits_{\Om\;\Om} \log \frac{1}{[z,w]}u(z)v(w)\dz\dw=\tau\int_\Om u(z)v(z)\dz,\quad \forall\, v\in L^2(\Om).
\end{equation}
The largest eigenvalue denoted by $\tau_h(\Om)$, of $\Lom_h$ (of \eqref{evproblem}) has the following variational characterization:
\begin{equation}\label{weightedcharmu1}
        \tau_h(\Om):=\sup\left\{ \frac{1}{2} \iint\limits_{\Om\;\Om} \log \frac{1}{[z,w]}u(z)u(w)\dz\dw: u\in L^2(\Om) \text{ and } \int_\Om u(z)^2\dz=1 \right\}.
\end{equation}

In \cite{Anoop-Ashok2023}, the authors established Faber-Krahn type inequalities for the first eigenvalue of the $p$-Laplacian with mixed boundary conditions under polarization. In the Euclidean setting, the polarizer $H$ is an affine half-space, that is, one side of a straight line in $\mathbb{R}^2$. The polarization $P_H(\Omega)$ of a bounded open set $\Omega$ with respect to $H$ is defined by
\begin{equation*}
P_H(\Omega)
= \bigl[(\Omega \cup \sigma_H(\Omega)) \cap H\bigr]
\cup \bigl[\Omega \cap \sigma_H(\Omega)\bigr],
\end{equation*}
where $\sigma_H$ is the reflection with respect to the boundary of $H$. The authors in \cite{anoopjiya2025} established the reverse Faber–Krahn inequality for the largest eigenvalue $\tau_1$ of \eqref{oldevproblem} under polarization: 
\begin{equation*}
         \tau_1(\Omega)\leq \tau_1(P_H(\Omega)).
    \end{equation*}

In the present work, we investigate the reverse Faber-Krahn inequality for the largest eigenvalue $\tau_h$ of \eqref{evproblem}, under polarization. The first step in our analysis is to introduce polarization in the hyperbolic space. In the hyperbolic space, straight lines are replaced by geodesics, and an affine half-space $\mathcal{H}$ (polarizer) corresponds to one side of a geodesic $\mathcal{G}$. For the next theorem, we extend the notion of polarization to the Poincar\'e disk model. Since this definition is not straightforward, it is presented in detail in Section~\ref{sec_def_pol}. To the best of our knowledge, this notion of polarization in hyperbolic space does not appear in the existing literature. Within this framework, we obtain the following theorem, which extends \cite[Theorem~1.6]{anoopjiya2025}.

\begin{theorem}\label{thm_hyperbolicfk}
Let $\Omega$ be an open bounded subset of Poincar\'e hyperbolic disk $\mathbb{D}$. Let $\mathcal{G}$ be a geodesic in $\mathbb{D}$ and let $\mathcal{H}$ be an associated polarizer. Then
\begin{equation}\label{pollog}
    \tau_h(\Omega) \le \tau_h\bigl(P_{\mathcal{H}}(\Omega)\bigr).
\end{equation}
Moreover, equality in \eqref{pollog} holds if and only if
\[
|P_{\mathcal{H}}(\Omega) \triangle \Omega|=0
\quad \text{or} \quad
|P_{\mathcal{H}}(\Omega) \triangle \sigma_{\mathcal{G}}(\Omega)|=0,
\]
where $\triangle$ denotes the symmetric difference of sets and $|\cdot|$ denotes the Lebesgue area measure.
\end{theorem}

In \cite[p.~367]{troutman1967}, the authors obtain a representation of the eigenfunctions associated with \eqref{oldevproblem} using tools from complex analysis. Let $\tau$ be an eigenvalue of $\mathcal{L}_h$ with corresponding eigenfunction $u$. If $\tau\neq 0$, then $u$ admits a continuous extension to $\mathbb{D}$ given by the same integral formula,
\begin{equation}\label{ef_on_D}
\tau u(z)
= \frac{1}{2} \int_\Omega \log\frac{1}{[z,w]}  u(w) \dw,
\quad \forall z \in \mathbb{D}.
\end{equation}
Following an approach similar to that of \cite{troutman1967}, we derive an analogous representation for eigenfunctions of problem~\eqref{evproblem}.
\begin{theorem}\label{thm_efreptheorem1}
    Let $\Omega \subset \mathbb{D}$ be an open set, bounded with respect to the hyperbolic metric $d_h$, and let $\tau$ be a nonzero eigenvalue of $\Lom_h$ satisfying \eqref{evproblem}, and $u$ be an associated eigenfunction. Then, for each $z\in\disk$ and for each $0<r<1-|z|$, the following formula holds:
    \begin{equation}\label{efrep1_hyper}
        \tau u(z)=\frac{\tau}{2\pi} \int_0^{2\pi} u(\phi_z(re^{i\theta}))d\theta-\frac{1}{2} \int_{\Om\cap \Delta_r(z)} \log \frac{[z,w]}{r}u(w)\dw,
    \end{equation}
    where $\Delta_r(z)=\{w:[z,w]<r\}=\{w:|\phi_z(w)|<r\}$. In addition, if $\tau = 0$, the above representation remains valid for any $z \in \Omega$ and any $r > 0$ such that $\Delta_r(z) \subset \Omega$.
\end{theorem}

Using the representation of eigenfunctions of \eqref{oldevproblem}, Troutman \cite[Corollary~1]{troutman1967} claimed that 0 is not an eigenvalue. However, the argument is incomplete, as it relies on the continuity of an eigenfunction corresponding to the eigenvalue 0, whereas only the continuity of $\Lom u$ is established. This ensures continuity only for eigenfunctions associated with nonzero eigenvalues. The gap can be remedied by using the Lebesgue differentiation theorem. Thus we  prove the following result, which is a consequence of the preceding theorem.
\begin{theorem}\label{cor_zero}
    Let $\Omega \subset \mathbb{D}$ be an open set, bounded with respect to the hyperbolic metric $d_h$. Then zero is not an eigenvalue of $\Lom_h$.
\end{theorem}
Additionally, in \cite{troutman1967,Kac1970}, it was shown that the operator $\mathcal{L}$ on $L^{2}(\Omega)$ is positive if and only if the transfinite diameter of $\Omega$ is less than or equal to~$1$. In the present setting, no such restriction is required, as we work in the Poincar\'e disk model. Using Theorem~\ref{thm_efreptheorem1}, we establish the positivity of the operator $\mathcal{L}_h$, as stated in the following theorem.

\begin{theorem}\label{thm_hyper_transfinite}
      Let $\Omega \subset \mathbb{D}$ be an open set, bounded with respect to the hyperbolic metric $d_h$. Then $\Lom_h $ is a positive operator on $L^2(\Om)$.
\end{theorem}

In our setting, the operator $\mathcal{L}_h$ is positive and therefore has no negative eigenvalues. In contrast, in the two-dimensional Euclidean case, the operator $\mathcal{L}$ admits a unique negative eigenvalue if and only if the transfinite diameter of $\Omega$ exceeds~$1$.


 \par This article is organized as follows. In Section \ref{sec_def_pol}, we introduce the definition of polarization on Poincar\'{e} disk model. In Section \ref{sec:main_theorems}, we prove the theorems stated in the Introduction.

 \section{Polarization on poincar\'e disk model}\label{sec_def_pol}
In this section, we recall the definition and basic properties of geodesics in the Poincar\'e disk model. We then define the reflection of a point $z \in \mathbb{D}$ with respect to a geodesic $\mathcal{G}$, and subsequently introduce the notion of polarization in the Poincar\'e disk model. Geodesics in the Poincar\'e disk are either straight lines through the origin or circles(Clines) orthogonal to $\mathbb{S}^1$\cite[Proposition A.5.6, p.26]{Benedetti_Carlo_1992}. We state the following definitions from \cite[Chapter 3]{hitchman}.
\begin{definition}{\bf (Cline).}
    A cline is a Euclidean line or circle, represented by the equation
    \begin{equation*}
        cz\ol{z}+\alpha z+ \ol{\alpha z}+d=0,
    \end{equation*}
    where $z = x + iy$ is a complex variable, $\alpha \in \mathbb{C}$, and $c, d \in \mathbb{R}$. The equation describes a line if $c = 0$, and a circle if $c \ne 0$ and $|\alpha|^2 > cd$.
\end{definition}
Two clines are said to be orthogonal if they intersect at right angles. For example, a line is orthogonal to a $\s^1$ if and only if it passes through the origin\cite[p.42]{hitchman}. We recall the following characterization \cite[Theorem 3.2.8]{hitchman}.
\begin{proposition}
    A cline through $z\neq 0\in \disk $ is orthogonal to $\s^1$ if and only if it passes through $z^*$, the point symmetric to $z$ with respect to $\s^1$, where $z^*$ is defined as
        \begin{equation*}
            z^*=\begin{cases}
                \frac{z}{|z|^2}, & z\neq 0,\\
                \infty, &z=0.
        \end{cases}
        \end{equation*}
\end{proposition}

Next, we define the reflection of a point in $\mathbb{D}$ with respect to a geodesic $\G$. 
Since the reflection with respect to the imaginary axis in $\disk$ can be easily defined, and the imaginary axis is symmetric about the $x$-axis, we begin by considering a geodesic $\G$ that is symmetric with respect to the $x$-axis. 
Accordingly, we take $\G$ to be the geodesic passing through $(a,0), \,0<a<1$ and orthogonal to $\mathbb{S}^1$.

\begin{center}
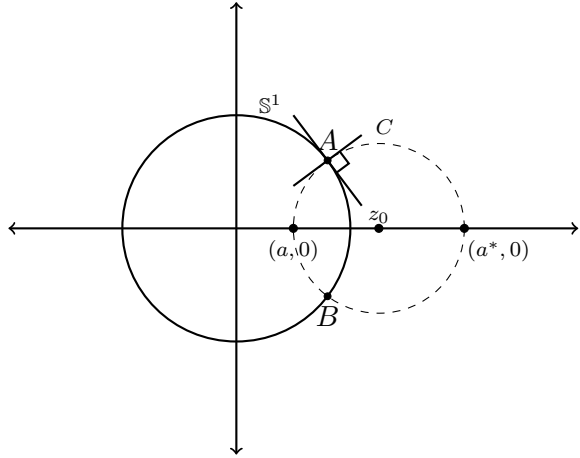

     \captionsetup{type=figure}
     \begin{tikzpicture}[scale=1.5]
         \path[name path=circleA] (0,0) circle(1);
         \draw[thick, name path=circleA] (0,0) circle(1);
         \path[name path=circleB] (1.25,0) circle(0.75);
         \draw[dashed, name path=circleB] (1.25,0) circle(0.75);

        \draw[<->, thick] (-2,0) -- (3,0);      
        \draw[<->, thick] (0,2) -- (0,-2);
        \draw[-, thick] (0.5,1) -- (1.1,0.2); 
        \draw[-, thick] (0.5,0.375) -- (1.1,0.825); 
        \filldraw[black] (2,0) circle (1pt);
        \filldraw[black] (0.5,0) circle (1pt);
        \filldraw[black] (1.25,0) circle (1pt);
        
          \path[name intersections={of=circleA and circleB, by={[label=above:$A$]P, [label=below:$B$]Q}}];
          \filldraw[black] (P) circle(0.03);
          \filldraw[black] (Q) circle(0.03); 

          \coordinate (A) at (1.1,0.825);
          \coordinate (B) at (0.8,0.6);
          \coordinate (C) at (1.1,0.2);
          \pic [draw, thick, angle radius=2mm] {right angle = A--B--C};
    
        \begin{scriptsize}    
            \draw (0.5,-0.2) node {$(a,0)$};
            \draw (2.3,-0.2) node {$(a^*,0)$};    
            \draw (1.25,0.1) node {$z_0$};
            \draw (0.3,1.1) node {$\s^1$};
            \draw (1.3,0.9) node {$C$};
        \end{scriptsize}
        
     \end{tikzpicture}
     \captionof{figure}{Geodesic passing through $(a,0)$}
 \end{center}
The dotted circle $C$\label{page_circle} is centered at $z_0=\frac{1+a^2}{2a}$ with Euclidean radius $r=\frac{1-a^2}{2a}$. It admits the parametrization $$z=\frac{1+a^2}{2a}+\frac{1-a^2}{2a}e^{i\theta}.$$ Then the geodesic $\G$ passing through the point $(a,0)$ is given by the intersection $$\G=C\cap \disk.$$ The points at which $C$ intersects the unit circle $\mathbb{S}^{1}$ are
\begin{align*}
    A&=\frac{1}{z_0} \left(1,\sqrt{z_0^2-1}\right)=\frac{2a}{a^2+1}\left(1,\frac{1-a^2}{2a}\right),\\
    B&=\frac{1}{z_0} \left(1,-\sqrt{z_0^2-1}\right)=\frac{2a}{a^2+1}\left(1,\frac{a^2-1}{2a}\right).
\end{align*}
Since reflection with respect to the imaginary axis is readily described in the Poincar\'e disk model, our approach is to map the geodesic $\G$ onto the imaginary axis by a M\"obius transformation. The reflection is then performed with respect to the imaginary axis and subsequently transferred back via the inverse M\"obius transformation. A natural choice of the M\"obius transformation $T$ is the one that maps the geodesic
\(\mathcal G\) onto the imaginary diameter $I_D$, that is
\[
T(\mathcal G)=\{\,iy : y\in\mathbb{R}\,\}\cap\mathbb{D}:=I_D,
\]
and satisfies
\[
T(A)=i,\; T(a)=0,\;\text{and} \; T(B)=-i.
\]
More precisely, the required M\"obius transformation is
\[
T(z)=\frac{z-a}{1-az}.
\]
We establish below the properties of \(T\) that are essential for defining the desired reflection.

\begin{proposition}\label{prop_map_T}
    Let $a\in(0,1)$. The M\"obius transformation $T(z)=\frac{z-a}{1-az}$ has the following properties:
    \begin{enumerate}[(i)]
        \item $T(a)=0, \, T(a^*)=\infty,\, T(A)=i$ and $T(B)=-i$.
        \item $T(C)$ is the imaginary axis.
        \item $T(\s^1)=\s^1$, $T(\disk)=\disk$ and $T(\ol{\disk}^c)=\ol{\disk}^c$.
        \item The interior of $C$, denoted by $U$, is mapped to the right half-plane $\{Re(z)>0\}$ and $\ol{U}^c$ is mapped to the left half-plane $\{Re(z)<0\}$.
        \item The region $U\cap \disk$ is mapped to the intersection of the unit disk with the right half-plane.
        \item The region $\ol{U}^c\cap \disk$ is mapped to the intersection of the unit disk with the left half-plane.
        \item\label{point_1} $T(\G)=I_D$.
    \end{enumerate}
\end{proposition}

 \begin{minipage}[t]{0.5\textwidth} 
     \begin{center}
     \captionsetup{type=figure}
     \begin{tikzpicture}[scale=1.25]
         \draw[dotted] (0,0) circle(1);
         \begin{scope}
        \clip (0,0) circle (1cm);
        \fill[dgray] (1.25,0) circle (0.75cm); 
        \draw[dotted](1.25,0) circle (0.75cm); 
        \end{scope}
        \draw[->](1.25,0) -- (1.75,0); 
        \begin{scope}
        \clip (2,-1) rectangle (3,1);
        \fill[dgray] (2,0) circle (1cm); 
        \draw[dotted](2,0) circle (1cm);
        \draw[dotted](2,-1) -- (2,1); 
        \end{scope}
     \end{tikzpicture}
     \captionof{figure}{$T(U\cap \disk)=\disk\cap \{x>0\}$}
 \end{center}
 \end{minipage}
 \begin{minipage}[t]{0.51\textwidth} 
 \begin{center}
 \captionsetup{type=figure}
     \begin{tikzpicture}[scale=1.25]
        \draw[dotted] (0,0) circle(1);
        \fill[dgray] (0,0) circle (1cm); 
        \begin{scope}
        \clip (0,0) circle (1cm);
        \fill[white] (1.25,0) circle (0.75cm); 
        \draw[dotted](1.25,0) circle (0.75cm); 
    \end{scope}
    \draw[->](1.25,0) -- (1.75,0); 
    \begin{scope}
        \clip (2,-1) rectangle (3,1);
        \fill[dgray] (3,0) circle (1cm); 
        \draw[dotted](3,0) circle (1cm);
        \draw[dotted](3,-1) -- (3,1); 
        \end{scope}
     \end{tikzpicture}
     \captionof{figure}{$T(\ol{U}^c\cap \disk)=\disk\cap \{x<0\}$}
 \end{center}
 \end{minipage}

\begin{proof}
\hfill\\
\begin{enumerate}[(i)]
    \item $T(a)=0, \, T(a^*)=\infty$, follows directly from the definition of $T$. Next, consider $T(A)=\frac{A-a}{1-aA}$. 
            A direct computation gives
            \begin{equation}
            A-a=\frac{1-a^{2}}{1+a^{2}}(a+i),
            \quad\text{and}\quad
            1-aA=\frac{1-a^{2}}{1+a^{2}}(1-ai).
            \end{equation}            
            Hence,
            \begin{equation}
            T(A)=\frac{a+i}{1-ai}=i.
            \end{equation}
            Similarly,
            \begin{equation}
            T(B)=\frac{a-i}{1+ai}=-i.
            \end{equation}
    \item \begin{align}
            T(z_0+re^{i\theta})&= \frac{z_0-a+re^{i\theta}}{1-z_0a -are^{i\theta}}=\frac{r+re^{i\theta}}{ar-are^{i\theta}}=\frac{1}{a}\cdot\,\frac{1+e^{i\theta}}{1-e^{i\theta}}\\
            &=\frac{1}{a}\cdot\,
\frac{e^{i\theta}-e^{-i\theta}}
{(1-\cos\theta)^2+\sin^2\theta}=i\frac{\sin\theta}{a(1-\cos\theta)}
          \end{align}
     \item We have 
            \begin{equation}
                |T(z)|^2=\frac{|z-a|^2}{|1-az|^2}.
            \end{equation}
          Consider
    \begin{align*}
        |1-az|^2 - |z-a|^2
        &= (1-2a \, Re(z)+a^2|z|^2)
           -(|z|^2-2a\, Re(z)+a^2) \\
        &= (1-a^2)(1-|z|^2).
    \end{align*}
    Hence,
    \begin{align*}
        |T(z)|<1 &\quad \text{if } |z|<1, \\
        |T(z)| = 1 &\quad \text{if } |z|=1, \\
        |T(z)| >1 &\quad \text{if } |z|>1.
    \end{align*}

    \item Since Möbius transformations are homeomorphisms, they preserve connected components. Moreover, $T(C) \text{ is the imaginary axis},$ so that \(T(U)\) and \(T(\overline{U}^{\,c})\) lie in the two complementary open half-planes determined by the imaginary axis. Furthermore,
        \[
        0\in \ol{U}^c \quad \text{and} \quad T(0)=-a,
        \]
        which lies on the left of the imaginary axis. Therefore,
        \[
        T(U)=\{Re(z)>0\}, \;\text{ and }\; T(\overline{U}^{\,c})\subset\{Re(z)<0\}.
        \]

    \item  and (vi) follow from the facts that 
             \[
        T(U\cap \disk)=T(U)\cap T(\disk),
        \quad 
        T(\overline{U}^{\,c}\cap \disk)=T(\overline{U}^{\,c})\cap T(\disk)\quad\text{and} \quad T(\disk)=\disk,
        \]
         together with part (iv).

\end{enumerate}   
\begin{enumerate}[(i)]
    \setcounter{enumi}{6}
    \item We have 
            \begin{equation}
                T(\G)=T(C\cap \disk)=T(C)\cap T(\disk)=\{\,iy : y\in\mathbb{R}\,\}\cap\mathbb{D}=I_D.
            \end{equation}
\end{enumerate}
\end{proof}

Using this M\"obius transformation, we define the reflection of a point with
respect to a geodesic \(\mathcal G\) as follows.

\begin{definition}\label{def_reflection}{\textbf{(Reflection of a point with respect to a geodesic $\G$)}.}
Geodesics in the Poincar\'e disk model are either diameters of the disk or arcs of circles that intersect the boundary $\mathbb{S}^1$ orthogonally \cite[Proposition A.5.6, p.26]{Benedetti_Carlo_1992}. We divide all geodesics in $\disk$ into three cases and define the associated reflections in each case. First, we consider geodesics that are straight lines passing through the origin. Second, we consider geodesics that are circular arcs symmetric with respect to the $x$-axis. Finally, we define reflection with respect to  arbitrary geodesics that are circular arcs. 
    \begin{itemize}
        \item Let $\G=L$ be a line passing through the origin making an angle $\theta$ ($<\pi$) with the real axis in the anti clockwise direction. The reflection of a point $z$ with respect to $L$, denoted by $\sigma_\G(z)$, is obtained by rotating $L$ to the imaginary axis, reflecting with respect to the imaginary axis, and rotating back. The sequence of transformations:
        \begin{equation}
            z\ra e^{i\left(\frac{\pi}{2}-\theta\right)}z= i e^{-i\theta}z\ra ie^{i\theta}\ol{z}\ra ie^{i\theta}\ol{z} e^{i\left(\theta-\frac{\pi}{2}\right)}=\ol{z}e^{2i\theta}.
        \end{equation}
        Thus,
        \begin{equation*}
            \sigma_\G(z)=\ol{z}e^{2i\theta}.
        \end{equation*}
        Equivalently, this reflection can be written as
        \begin{equation*}
            \sigma_\G(z)
            = e^{-i\theta'}\, I^{-1}\!\left(\overline{-\,I\!\left(e^{i\theta'}z\right)}\right),
        \end{equation*}
        where $I$ denotes the identity map and $\theta'=\tfrac{\pi}{2}-\theta$.

        \item Let $\G$ be a geodesic passing through $(a, 0)$ with $0<a<1$, and symmetric with respect to $x$-axis. To define reflection with respect to $\G$, we use the M\"obius transformation $T$ that maps $\G$ to the imaginary diameter: 
        \begin{equation*}
            T(z)=\frac{z-a}{1-az}.
        \end{equation*}        
        Then, by reflecting with respect to the imaginary axis and by taking inverse M\"obius transformation, we get the required reflection with respect to $\G$:
        \begin{equation*}
            z\ra T(z)\ra\ol{-T(z)}\ra T^{-1}(\ol{-T(z)}) .
        \end{equation*}
        More precisely, 
        \begin{equation*}
            \sigma_\G(z)=T^{-1}(\ol{-T(z)}) .
        \end{equation*}

        \begin{center}
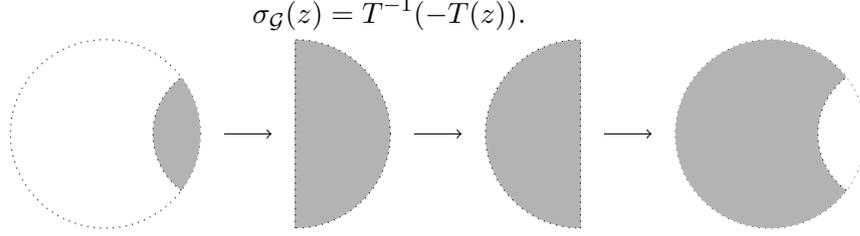

     \captionsetup{type=figure}
     \begin{tikzpicture}[scale=1.25]
         \draw[dotted] (0,0) circle(1);
         \begin{scope}
        \clip (0,0) circle (1cm);
        \fill[dgray] (1.25,0) circle (0.75cm); 
        \draw[dotted](1.25,0) circle (0.75cm); 
        \end{scope}
        
        \draw[->](1.25,0) -- (1.75,0); 
        
        \begin{scope}
        \clip (2,-1) rectangle (3,1);
        \fill[dgray] (2,0) circle (1cm); 
        \draw[dotted](2,0) circle (1cm);
        \draw[dotted](2,-1) -- (2,1); 
        \end{scope}
        
        \draw[->](3.25,0) -- (3.75,0); 

        \begin{scope}
        \clip (4,-1) rectangle (5,1);
        \fill[dgray] (5,0) circle (1cm); 
        \draw[dotted](5,0) circle (1cm);
        \draw[dotted](5,-1) -- (5,1); 
        \end{scope}
        
        \draw[->](5.25,0) -- (5.75,0); 
        
        \draw[dotted] (7,0) circle(1);
        \fill[dgray] (7,0) circle (1cm); 
        \begin{scope}
        \clip (7,0) circle (1cm);
        \fill[white] (8.25,0) circle (0.75cm); 
        \draw[dotted](8.25,0) circle (0.75cm); 
        \end{scope}
   
     \end{tikzpicture}
     \captionof{figure}{Reflection of a region}
 \end{center}

        \item Let $\G$ be an arbitrary geodesic in $\disk$. Choose an angle $\theta$ such that $e^{i\theta}\G$ is symmetric with respect to the real axis, and $\G \cap \{\text{real axis}\} = \{(a, 0)\}$, where $a>0$. Then proceed with the M\"obius transformation, reflection, and the inverse M\"obius transformation as in the second case, and then rotate back:
        \begin{equation*}
            z\ra e^{i\theta}z \ra T(e^{i\theta}z)\ra\ol{-T(e^{i\theta}z)}\ra T^{-1}(\ol{-T(e^{i\theta}z)})\ra e^{-i\theta} T^{-1}(\ol{-T(e^{i\theta}z)}).
        \end{equation*}
        In this case,
        \begin{equation*}
            \sigma_\G(z)=e^{-i\theta} T^{-1}(\ol{-T(e^{i\theta}z)}).
        \end{equation*}
    \end{itemize}
\end{definition}
Next, we introduce the definition of polarization using the notion of reflection described above.
\begin{definition}{\textbf{(Polarization).}}
A \emph{polarizer} $\H$ corresponds to one side of a geodesic $\G$. Let $\sigma_\G$ denote the reflection with respect to the geodesic $\G$. For a set $\Omega \subset \disk$, the \emph{polarization} of $\Omega$ with respect to $\H$, denoted by $P_\H(\Omega)$, is defined as:
\begin{equation*}
    P_\H(\Omega) = \big[(\Omega \cup \sigma_\G(\Omega)) \cap \H\big] \cup \big[\Omega \cap \sigma_\G(\Omega)\big].
\end{equation*}
\end{definition}

\begin{definition}{\textbf{(Polarization of a function).}}
    For a measurable function $f:\disk \rightarrow \mathbb{R}$, the polarization $P_\H(f)$ with respect to $\H$ is defined as 
   \begin{equation*}
       P_\H(f)(z)=\begin{cases}
           \max\{f(z),f(\sigma_\G(z))\},& \text{   for }z\in \H,\\
           \min\{f(z),f(\sigma_\G(z))\},& \text{   for }z\in \mathbb{D}\setminus \H.
       \end{cases}
   \end{equation*}
   Let $f:\Omega\rightarrow \mathbb{R}$ and let $\Tilde{f}$ be its zero extension to $\disk$. The polarization $P_\H(f)$ is defined as the restriction of $P_\H(\Tilde{f})$ to $P_\H(\Omega)$.
\end{definition} 
Next, we establish the following propositions, which are required for proving the main theorems stated in the Introduction.
\begin{proposition}\label{prop_mobius_preseve_dist}
    Let $\psi$ be a  M\"obius transformation of $\disk$ onto $\disk$, and let $z,w\in\disk$. Then, 
    \begin{equation}
        [z,w]=[\psi(z),\psi(w)].
    \end{equation}
\end{proposition}

\begin{proof}
   \par For $z,w\in\disk$, the hyperbolic distance $d_h(z,w)$ between $z$ and $w$ is
\begin{equation*}
    d_h(z,w)=2\tanh^{-1}\frac{|z-w|}{|1-\ol{z}w|},
\end{equation*}
follows from \cite[p.11]{Stoll_2016}. The M\"{o}bius transformation preserves the hyperbolic distance, as proved in \cite[Theorem 2.2.1, p. 10]{Stoll_2016}. Consequently, $d_h(z,w)=d_h(\psi(z),\psi(w))$, and hence $2\tanh^{-1}[z,w]=2\tanh^{-1}[\psi(z),\psi(w)]$. Hence, the pseudo-hyperbolic distance is invariant under $\psi$, that is  $[z,w]=[\psi(z),\psi(w)]$.
\end{proof}

It is known that if two points $z,w$ lie on the same side of the imaginary axis, then their Euclidean distance satisfies
\[
|z-w|<|z-(-\overline{w})|,
\]
where $-\overline{w}$ denotes the reflection of $w$ with respect to the imaginary axis. In the following proposition, we show that the above inequality is preserved for the pseudohyperbolic distance as well.

\begin{proposition}\label{prop_refle_dist}
Let \(z=a+ib\) and \(w=c+id\) be points in \(\disk\) that lie in the same open half-disk of \(\disk\) determined by the \(y\)-axis, that is,
\(\{\zeta\in\disk : Re (\zeta)>0\}\) or \(\{\zeta\in\disk : Re (\zeta)<0\}\). 
Then
\begin{equation}
        [z,w]<[z,-\ol{w}]. 
    \end{equation}
\end{proposition}
\begin{proof}
From \eqref{eqn_phi_z_w}, one can observe \( [z,w]=|\phi_z(w)| \) and 
\begin{equation}
1-[z,w]^2
= 1-|\phi_z(w)|^2
= \frac{|1-\overline{z}w|^2-|z-w|^2}{|1-\overline{z}w|^2}
= \frac{(1-|z|^2)(1-|w|^2)}{|1-\overline{z}w|^2}.
\end{equation}
Similarly, we have
\begin{equation}
1-[z,-\overline{w}]^2
= 1-|\phi_z(-\overline{w})|^2
= \frac{(1-|z|^2)(1-|w|^2)}{|1+\overline{z}\,\overline{w}|^2}.
\end{equation}
Furthermore,
\begin{align}
|1+\overline{z}\,\overline{w}|^2 - |1-\overline{z}w|^2
&= \bigl(1+|z|^2|w|^2+2\,Re(zw)\bigr)
 - \bigl(1+|z|^2|w|^2-2\,Re(z\overline{w})\bigr) \nonumber\\
&= 2\,Re(zw+z\overline{w})
= 2ac > 0,
\end{align}
since \(z=a+ib\) and \(w=c+id\) lie in the same open half-disk of \(\disk\) determined by the \(y\)-axis. Consequently, $1-[z,-\overline{w}]^2 < 1-[z,w]^2,$
which implies $[z,w] < [z,-\overline{w}].$
\end{proof}

\begin{proposition}\label{prop_reflection}
Let $z, w \in \disk$, and let $\G$ be a geodesic in $\disk$ (either a straight line passing through origin or a circle orthogonal to $\s^1$). Then the reflection $\sigma_\G$ defined in Definition~\ref{def_reflection} has the following properties:
\begin{enumerate}[(i)]
    \item $\sigma_\G\circ \sigma_\G=I$, where $I$ denotes the identity map on $\disk$.
    \item\label{point_2} $\sigma_\G=I$, on $\G$.
    \item $[z, w] = [\sigma_\G(z), \sigma_\G(w)]$.
    \item $[z, \sigma_\G(w)] = [\sigma_\G(z), w]$.
    \item If $z$ and $w$ lie on the same side of $\G$, then $[z, w] < [z, \sigma_\G(w)]$.
\end{enumerate}
\end{proposition}
\begin{proof}
    If the geodesic is a straight line passing through origin, then the  above properties are immediate. Hence, we prove the case where the geodesic $\G$ is a circle orthogonal to $\mathbb{S}^1$. Choose an angle $\theta$ such that $e^{i\theta}\G$ is symmetric with respect to the real axis, and $e^{i\theta}\G \cap \{\text{real axis}\} = \{(a, 0)\}$, where $a>0$. The reflection across $\G$ is given by 
        \begin{equation*}
            \sigma_\G(z)=e^{-i\theta} T^{-1}(\ol{-T(e^{i\theta}z)}),\quad\text{where }\quad T(z)=\frac{z-a}{1-az}.
        \end{equation*}

    \begin{enumerate}[(i)]
        \item \begin{align}
            \sigma_\G\circ \sigma_\G(z)&=e^{-i\theta} T^{-1}(\ol{-T(e^{i\theta}\sigma_\G(z))})=e^{-i\theta} T^{-1}(\ol{-T( T^{-1}(\ol{-T(e^{i\theta}z)}))})\\
            &=e^{-i\theta} T^{-1}(\ol{-(\ol{-T(e^{i\theta}z)})})=e^{-i\theta} T^{-1}(T(e^{i\theta}z))=z.
        \end{align}

        \item Let $\mathcal G':=e^{i\theta}\G = C \cap \mathbb D$, where $C$ is the Euclidean circle with centre $z_0 = \frac{1+a^2}{2a}$ and radius $r = \frac{1-a^2}{2a}$, as defined on Page \pageref{page_circle}. Let $z \in \mathcal G$ and let $w = e^{i\theta} z$. Then $w$ can be written in the form $w = z_0 + r e^{i\eta}$. Recall that 
        \[
        T(w) = \frac{w-a}{1-aw},
        \;\text{ and }\;
        T^{-1}(w) = \frac{w+a}{1+aw}.
        \]
A direct computation yields
\begin{align*}
T^{-1}\!\left(-\overline{T(w)}\right)
&= \frac{\left(\dfrac{a-\overline{w}}{1-a\overline{w}}\right) + a}
        {1 + a \left(\dfrac{a-\overline{w}}{1-a\overline{w}}\right)} = \frac{a-\overline{w} + a(1-a\overline{w})}
        {1-a\overline{w} + a(a-\overline{w})}\\
        &= \frac{2a - \overline{w}(1+a^2)}
        {(1+a^2) - 2a\overline{w}}
= \frac{1 - \overline{w} z_0}{z_0 - \overline{w}}.
\end{align*}
Since $z_0$ is real, it follows that
\begin{equation}
    \overline{w} = z_0 + r e^{-i\eta},\text{ and }z_0 - \overline{w} = -r e^{-i\eta}.
\end{equation}
Therefore,
\begin{align*}
T^{-1}\!\left(-\overline{T(w)}\right)
&= \frac{1 - z_0(z_0 + r e^{-i\eta})}{-r e^{-i\eta}} = \frac{z_0^2 - 1 + z_0 r e^{-i\eta}}{r e^{-i\eta}}.
\end{align*}
Using
\[
z_0^2 - 1
= \left(\frac{1+a^2}{2a}\right)^2 - 1
= \left(\frac{1-a^2}{2a}\right)^2
= r^2,
\]
we obtain
\[
T^{-1}\!\left(-\overline{T(w)}\right)
= z_0 + r e^{i\eta}=w=e^{i\theta }z.
\]
Hence,
\[
\sigma_\G(z)=e^{-i\theta} T^{-1}(\ol{-T(e^{i\theta}z)})=e^{-i\theta} e^{i\theta }z=z
\]

        \item The map $\sigma_\G$ is obtained as a composition of M\"obius transformations and a reflection. Hence, by Proposition~\ref{prop_mobius_preseve_dist}, together with the invariance properties $$[z,w]=[-z,-w]=[\ol{z},\ol{w}],$$ it follows that 
            \begin{align*}
            [\sigma_{\mathcal G}(z),\sigma_{\mathcal G}(w)]
            &= \bigl[e^{-i\theta} T^{-1}\!\bigl(-\,\overline{T(e^{i\theta}z)}\bigr),
                    e^{-i\theta} T^{-1}\!\bigl(-\,\overline{T(e^{i\theta}w)}\bigr)\bigr] \\
            &= \bigl[-\,\overline{T(e^{i\theta}z)},-\,\overline{T(e^{i\theta}w)}\bigr] = \bigl[T(e^{i\theta}z),T(e^{i\theta}w)\bigr] = [z,w].
            \end{align*}

        \item This follows directly from $(i)$ and $(iii)$:
                \begin{align}
                [z,\sigma_{\G}(w)]
                = [\sigma_{\G}(z),\sigma_{\G}(\sigma_{\G}(w))]
                = [\sigma_{\G}(z),w].
                \end{align}


        \item Let $z$ and $w$ lie on the same side of $\G$. Then $T(e^{i\theta} z)$ and $T(e^{i\theta} w)$ lie on the same side of the corresponding half-disk of $\mathbb{D}$, determined by the imaginary axis. Then from Proposition \ref{prop_refle_dist}, it follows that
            \begin{align}
                [z, \sigma_\G(w)]&=[z,e^{-i\theta} T^{-1}(\ol{-T(e^{i\theta}w)})]=[T(e^{i\theta}z),\ol{-T(e^{i\theta}w)}]\\
                &>[T(e^{i\theta}z),T(e^{i\theta}w)]=[z,w].
            \end{align}
    \end{enumerate}
\end{proof}
In $\R^2$, let $H$ be a polarizer and $\partial H$ be its boundary line.
The reflection with respect to $\partial H$ can be described geometrically as follows.
Given a point $x\in H$, consider the line through $x$ that is orthogonal to $\partial H$,
and let $y$ be its intersection point with $\partial H$.
There exists a unique point $x'\in \ol{H}^c$ on this orthogonal line such that
\[
|x-y| = |x'-y|.
\]
This point $x'$ is precisely the reflection of $x$ with respect to $\partial H$.
An analogous interpretation holds in the Poincar\'e disk model:
reflection across a geodesic can be obtained by moving points along geodesics
orthogonal to the given geodesic while preserving pseudo-hyperbolic distance.
This interpretation of $\sigma_\G$ is established in the following proposition.

\begin{proposition}
Let $\mathcal G$ and $\mathcal G'$ be geodesics in $\mathbb{D}$ such that $\mathcal G$ and $\mathcal G'$ intersect orthogonally. Then the following hold:
\begin{enumerate}[(i)]
    \item $\sigma_{\mathcal G}(\mathcal G') = \mathcal G'$.
    \item $[\sigma_{\mathcal G}(z), w] = [z, w]$, for each $z \in \disk$ and $w\in\G$.
\end{enumerate}

\begin{proof}
\hfill\\
    \begin{enumerate}[(i)]    
        \item Let $\G\cap\G'=\{w\}$.  Choose an angle $\theta$ such that $e^{i\theta}\G$ is symmetric with respect to the real axis and $e^{i\theta}\G \cap \{\text{real axis}\} = \{(a, 0)\}$, where $a>0$. Since $\mathcal G \perp \mathcal G'$, and Möbius transformations preserve angles, it follows that $T\!\left(e^{i\theta}\mathcal G\right) \perp T\!\left(e^{i\theta}\mathcal G'\right)$. By Proposition~\ref{prop_map_T}, part~\ref{point_1}, we have that
$T\!\left(e^{i\theta}\mathcal G\right)$ is the imaginary diameter, $I_D$.
Since a circle orthogonal to a line must have its centre on that line,
it follows that $T\!\left(e^{i\theta}\mathcal G'\right)$ is a circle with centre on the imaginary axis.
Consequently, $T\!\left(e^{i\theta}\mathcal G'\right)$ is symmetric with respect to the imaginary axis, and hence    $T\!\left(e^{i\theta}\mathcal G'\right)=-\overline{T\!\left(e^{i\theta}\mathcal G'\right)}$. Applying $T^{-1}$ and multiplying by $e^{-i\theta}$, we obtain
\[
e^{-i\theta} T^{-1}\!\left(T\!\left(e^{i\theta}\mathcal G\right)\right)
\perp
e^{-i\theta} T^{-1}\!\left(-\overline{T\!\left(e^{i\theta}\mathcal G'\right)}\right), \;\text{ that is }\;
\mathcal G \perp \sigma_{\mathcal G}(\mathcal G').
\]
On the other hand, we already know that $\mathcal G \perp \mathcal G'$ and that
$w = \sigma_{\mathcal G}(w) \in \mathcal G \cap \mathcal G' \cap \sigma_{\mathcal G}(\mathcal G')$
(see Proposition~\ref{prop_reflection}, part~\ref{point_2}). Let $\phi$ be the M\"obius transformation of $\mathbb D$ such that
\[
\phi(\mathcal G) = \text{Imaginary axis}\cap \disk=I_D,
\quad\text{and}\quad 
\phi(w) = 0.
\]
Then the image geodesics $\phi(\mathcal G')$ and
$\phi\!\left(\sigma_{\mathcal G}(\mathcal G')\right)$ are geodesics in $\mathbb D$
passing through $0$ and orthogonal to the imaginary diameter $I_D=\{0+ai:-1<a<1\}$.
In the Poincaré disk, there is a unique geodesic through $0$ orthogonal to $I_D$,
namely the real diameter $R_D=\{a+0i:-1<a<1\}$. Hence,
\[
\phi(\mathcal G')
=
\phi\!\left(\sigma_{\mathcal G}(\mathcal G')\right)
=
R_D.
\]
Applying $\phi^{-1}$, we conclude that
\[
\mathcal G' = \sigma_{\mathcal G}(\mathcal G').
\]

        \item Since $w \in \mathcal G$ and $\sigma_{\mathcal G}(w)=w$ (see \ref{point_2} of Proposition \ref{prop_reflection}), it follows that
        \begin{equation}
            [\sigma_{\mathcal G}(z), w]=[z,\sigma_\G(w)]=[z,w] .
        \end{equation}
    \end{enumerate}
\end{proof}
\end{proposition}
\begin{proposition}
    Let $\mathcal{G}$ be a geodesic in $\mathbb{D}$. Then the modulus of determinant of Jacobian of the reflection map $\sigma_\G$ is 
    \begin{equation}
        |J_{\sigma_\G}(z)|=\frac{|1-|\sigma_\G(z)|^2|^2}{|1-|z|^2|^2}.
    \end{equation}
\end{proposition}
\begin{proof}
    By the chain rule for the Jacobian and the fact that rotations and reflections have Jacobian determinant of modulus one, we obtain
    \begin{align}
        |J_{\sigma_\G}(z)|&=|J_{T^{-1}}(\ol{-T(e^{i\theta}z)})|\cdot|J_{-T}(e^{i\theta }z)|.
    \end{align}
    For a M\"{o}bius transformation $\psi$, recall from \cite[Theorem 3.3.1, p.23]{Stoll_2016} that
    \begin{equation}\label{eqn_det_jac_mob}
       |J_{\psi}(z)|=\frac{|1-|\psi(z)|^2|^2}{|1-|z|^2|^2}.
    \end{equation}
    Therefore, 
    \begin{align}\label{eqn_det_jac_sigma_g}
        |J_{\sigma_\G}(z)|&=\frac{|1-|T^{-1}(\ol{-T(e^{i\theta}z)})|^2|^2}{|1-|\ol{-T(e^{i\theta}z)}|^2|^2}\cdot \frac{|1-|-T(e^{i\theta}z)|^2|^2}{|1-|e^{i\theta}z|^2|^2}\nonumber\\
        &=\frac{|1-|T^{-1}(\ol{-T(e^{i\theta}z)})|^2|^2}{|1-|z|^2|^2}=\frac{|1-|\sigma_\G(z)|^2|^2}{|1-|z|^2|^2}.
    \end{align}
\end{proof}

\begin{proposition}\label{prop_int_reflection}
    Let $E\subseteq \disk$, $f\in L^2(\disk)$, and let $\mathcal{G}$ be a geodesic in $\mathbb{D}$. Then,
     \begin{align}
        \int_{\sigma_\G(E)}f(z)\dz=\int_{E}f(\sigma_\G (z)) \dz.
    \end{align}
\end{proposition}
\begin{proof}
     \begin{align}
        \int_{\sigma_\G(E)}f(z)\dz&=\frac{1}{\pi}\int_{\sigma_\G(E)}f(z)  \frac{dA(z)}{(1-|z|^2)^2}\\
        &=\frac{1}{\pi}\int_{E} \frac{f(\sigma_\G (w))}{(1-|\sigma_\G (w))|^2}|J_{\sigma_\G}(w)|dA(w)\\
        &= \frac{1}{\pi}\int_{E}f(\sigma_\G (w)) \frac{dA(w)}{(1-|w|^2)^2}\quad\text{ (follows from }\eqref{eqn_det_jac_sigma_g})\\
        &=\int_{E}f(\sigma_\G (z)) \dz.
    \end{align}
\end{proof}

\begin{remark}
    Let $\mathcal{G}$ be a geodesic in $\mathbb{D}$ and let $\mathcal{H}$ be an associated polarizer. Define 
    $\widetilde{\H}: = \sigma_\G(\H)$. Then 
    \begin{equation}
        \disk=\H\sqcup\widetilde{\H}\sqcup\G\text{ and }\disk\setminus\H= \widetilde{\H}\sqcup\G.
    \end{equation}    
     For integrals over $\disk\setminus\H$, it suffices to integrate over $\widetilde{H}$, since $\G$ has zero hyperbolic and Lebesgue measure.
\end{remark}

\begin{proposition}\label{prop_pol_properties}
Let $f\in L^{2}(\disk)$, and let $E \subset \mathbb{D}$ be a measurable set. Let $\mathcal{G}$ be a geodesic in $\mathbb{D}$ and let $\mathcal{H}$ be an associated polarizer. Then the following properties hold:
\begin{enumerate}[(i)]
    \item
    \begin{equation}
                    \int_\disk (P_\H(f(z)))^2\dz= \int_\disk (f(z))^2\dz.
                \end{equation}

    \item
    If $f\ge 0$ and $f=0$ on $E^{c}$, then $P_{\mathcal{H}}f=0$ on $P_{\mathcal{H}}(E)^{c}$.

    \item
    If $f\ge 0$, then
     \begin{equation}
                    \int_{P_\H(E)} P_\H(f(z))^2\dz= \int_E f(z)^2\dz.
                \end{equation}
\end{enumerate}
\end{proposition}
\begin{proof}
\hfill\\
    \begin{enumerate}[(i)]
        \item We have
    \begin{align}
        \int_\disk P_\H(f(z))^2\dz&=\int_\H P_\H(f(z))^2\dz+\int_{\disk\setminus\H} P_\H(f(z))^2\dz
    \end{align}
    For any $f\in L^2(E)$, we have
    \begin{align}
        \int_{\widetilde{\H}}f(z)^2\dz=\int_{\H}f(\sigma_\G (z))^2 \dz,
    \end{align}
    which follows from Proposition~\ref{prop_int_reflection}.
    Thus
    \begin{align}
        \int_\disk P_\H(f(z))^2\dz&=\int_\H P_\H(f(z))^2\dz+\int_{\H} P_\H(f(\sigma_\G (z)))^2\dz\\
        &=\int_\H [P_\H(f(z))^2+ P_\H(f(\sigma_\G (z)))^2]\dz\\
        &=\int_\H [(f(z))^2+ (f(\sigma_\G( z)))^2]\dz=\int_\disk f(z)^2\dz
    \end{align}
    The intermediate equality follows from the definition of polarization: for each
\(z \in \H\), one of \(P_{\H}(f(z))\) and \(P_{\H}(f(\sigma_{\G}(z)))\) equals
\(f(z)\), while the other equals \(f(\sigma_{\G}(z))\).

        \item Observe that \begin{equation*}
                P_\H(E)^c=\left[\H\setminus (E\cup\sigma_\G(E))\right] \cup [(\disk\setminus\H)\setminus (E\cap \sigma_\G(E))].
                \end{equation*}
            For $z\in \H\setminus (E\cup\sigma_\G(E))$, $f(z)=f(\sigma_\G(z))=0.$ Thus $P_\H(f)(z)=0$. For $z\in (\disk\setminus\H)\setminus (E\cap \sigma_\G(E))$, we have either $f(z)=0$ or $f(\sigma_\G(z))=0$. Therefore, $P_\H(f)(z)=\min \{f(z),f(\sigma_\G(z))\}=0$, since $f$ is non-negative. Thus, $P_\H(f)=0$ on $P_\H(E)^c$.

        \item Follows from (i) and (ii).
    \end{enumerate}
\end{proof}

 \section{Proof of the main theorems}\label{sec:main_theorems}
In this section, we prove that the operator $\Lom_h$ is compact and self-adjoint (Proposition~\ref{prop_L_h_compact_sa}). We establish the positivity of the eigenfunction associated with the largest eigenvalue $\tau_h$ of $\Lom_h$ (Proposition~\ref{prop-non-neg}). We also derive Riesz-type inequality under polarization (Proposition~\ref{prop_riezpol}). We then prove the main results stated in the Introduction, namely: the reverse Faber–Krahn inequality (Theorem~\ref{thm_hyperbolicfk}); the representation theorem for eigenfunctions (Theorem~\ref{thm_efreptheorem1}); the fact that zero is not an eigenvalue (Theorem~\ref{cor_zero}); and the positivity of the operator $\mathcal{L}_h$ (Theorem~\ref{thm_hyper_transfinite}). Throughout this section, for any $z,w \in \mathbb{D}$, we denote  $\phi_z$, for the M\"{o}bius transformation that takes unit disk $\disk$ onto itself defined by
\begin{equation*}
\phi_z(w)=\frac{z-w}{1-\overline{z}w}.
\end{equation*}
Note that $|\phi_z(w)| = [z,w], \, \, \phi_z(\phi_z(w))=w,\,\forall\, w\in \disk$, and $ \phi_z(z)=0$.

\begin{proposition}\label{prop_L_h_compact_sa}
    Let $\Om\subset \disk$ be a bounded open subset of $\disk$ with respect to the hyperbolic metric on $\disk$. The operator $\Lom_h : L^2(\Omega) \to C(\overline{\Omega}) \subset L^2(\Omega)$ is compact and self-adjoint on $L^2(\Omega)$.
\end{proposition}
\begin{proof}
    For $f\in L^2(\Om)$, it can be shown that $\Lom_h f\in L^2(\Om)$:
\begin{align}\label{uniform_bound}
    |\Lom_h f(z)|^2&=\frac{1}{4}\left| \int_\Om \log \frac{1}{[z,w]}f(w)\dw\right|^2\nonumber\\
    &\leq \frac{\|f\|_{L^2(\Om)}^2}{4} \int_\Om \left| \log \frac{1}{[z,w]}\right|^2\dw\nonumber\\
    &\leq \frac{\|f\|_{L^2(\Om)}^2}{4} \int_\disk \left| \log \frac{1}{|\phi_z(w)|}\right|^2\dw\nonumber\\
    &=  \frac{\|f\|_{L^2(\Om)}^2}{4\pi} \int_\disk \left| \log \frac{1}{|\widetilde{w}|}\right|^2  |J_{\phi_z}(\widetilde{w})| \frac{dA(\widetilde{w})}{(1-|\phi_z(\widetilde{w})|^2)^2}          \nonumber\\
    &= \frac{\|f\|_{L^2(\Om)}^2}{4\pi} \int_\disk \left| \log \frac{1}{|w|}\right|^2\frac{dA(w)}{(1-|w|^2)^2}\qquad\text{ (from }\eqref{eqn_det_jac_mob})
\end{align}
From the Taylor series expansion, we have 
\begin{equation*}
    \frac{1}{(1-r^2)^2} =\sum_{n=0}^\infty (n+1)r^{2n},\quad |r|<1
\end{equation*}
\begin{equation}
    \int_0^1 \frac{(\log\, r)^2 r}{(1-r^2)^2}dr=\sum_{n=0}^\infty (n+1) \int_0^1 r^{2n+1}(log\, r)^2 dr=\sum_{n=0}^\infty (n+1)\frac{1}{4(n+1)^3}=\frac{\pi^2}{24}.
\end{equation}
Substituting back in \eqref{uniform_bound} we get
\begin{equation}\label{finalbound}
    |\Lom_h f(z)|^2 \leq \frac{\|f\|_{L^2(\Om)}^2}{4\pi} \cdot \frac{2\pi\cdot\pi^2}{24}=\frac{ \pi^2 \|f\|_{L^2(\Om)}^2}{48}.
\end{equation}
Since $\Omega$ is bounded with respect to the hyperbolic metric, there exists a constant $M>0$ such that, for every $z\in\Omega$,
\[
d_h(0,z)=\log\!\left(\frac{1+|z|}{1-|z|}\right)\le M.
\]
Consequently,
\[
\frac{1}{1-|z|}\le \frac{1+|z|}{1-|z|}\le e^{M},
\]
and hence
\[
|z|\le 1-\frac{1}{e^{M}}=:k<1.
\]
Together with \eqref{finalbound}, we have
\begin{equation}
\int_{\Omega} |\mathcal{L}_h f(z)|^2\dz
\le \frac{\pi}{48}\,\|f\|_{L^2(\Omega)}^{2}
\int_{\Omega} \frac{dA(z)}{(1-|z|^2)^2}
\le \frac{\pi}{48}\,
\frac{|\Om|}{(1-k^2)^2}\,\|f\|_{L^2(\Omega)}^{2},
\end{equation}
where \(|\Om|\) denotes the Lebesgue area measure of \(\Omega\), which is finite. Therefore, $\Lom_h f\in L^2(\Om)$, and the operator $\Lom_h$ is a bounded operator on $L^2(\Om)$. 
For any $z_1, z_2\in\Om$, we have
\begin{align}\label{uniform_continuity_H}
    |\Lom_h f(z_1)-\Lom_h f(z_2)|^2&\leq \frac{\|f\|_{L^2(\Om)}^2}{4} \left[\int_\Om \left| \log \frac{[z_2,w]}{[z_1,w]}\right|^2 \dw\right]\nonumber\\
    &\leq \frac{\|f\|_{L^2(\Om)}^2}{2} \left[\int_\Om \left| \log \frac{|z_2-w|}{|z_1-w|}\right|^2 \dw + \int_\Om \left| \log \frac{|1-\ol{z_1}w|}{|1-\ol{z_2}w|}\right|^2 \dw\right]\nonumber\\
\end{align}
To estimate the integrals, we use the following inequality:
\begin{equation}\label{log_inequality}
    \log(1 + t) \leq \frac{t^\alpha}{\alpha}, \quad \forall\, t > 0, \text{ and } \alpha \in (0,1).
\end{equation}
Define
\begin{align}
F(z_1,z_2,w)
&:= \frac{|z_2-w|}{|z_1-w|}.
\end{align}
Then,
\begin{align}
\left| \log F(z_1,z_2,w) \right|
&= \left| \log F(z_2,z_1,w) \right|.
\end{align}
Let
\begin{align}
\Omega_1 
&:= \left\{ w \in \Om : F(z_1,z_2,w) \ge 1 \right\},
\;\text{ and }\;
\Omega_2 := \left\{ w \in \Om : F(z_1,z_2,w) < 1 \right\}.
\end{align}
The first integral in \eqref{uniform_continuity_H} can be estimated as follows:
\begin{align}\label{prop1_e1}
\int_\Om 
\left| \log \frac{|z_2-w|}{|z_1-w|} \right|^2 \dw
&=
\int_{\Omega_1}
\left| \log \frac{|z_2-w|}{|z_1-w|} \right|^2 \dw +
\int_{\Omega_2}
\left| \log \frac{|z_1-w|}{|z_2-w|} \right|^2 \dw .
\end{align}
For \(w \in \Omega_1\), let
\begin{align}
t_w
:= \frac{|z_2-w|}{|z_1-w|}-1
= \frac{|z_2-w|-|z_1-w|}{|z_1-w|}
\ge 0.
\end{align}
Using \eqref{log_inequality}, we obtain
\begin{align}
0
\le \log \frac{|z_2-w|}{|z_1-w|}
&\le \frac{1}{\alpha}
\left( \frac{|z_2-w|-|z_1-w|}{|z_1-w|} \right)^{\alpha}
\\
&\le \frac{1}{\alpha}
\frac{|z_1-z_2|^{\alpha}}{|z_1-w|^{\alpha}}.
\end{align}
Consequently,
\begin{align}
\left| \log \frac{|z_2-w|}{|z_1-w|} \right|^2
\le \frac{1}{\alpha^2}
\frac{|z_1-z_2|^{2\alpha}}{|z_1-w|^{2\alpha}}.
\end{align}
Integrating over \(\Omega_1\) yields
\begin{align}\label{prop1_e2}
\int_{\Omega_1}
\left| \log \frac{|z_2-w|}{|z_1-w|} \right|^2 \, \dw
\le
\frac{|z_1-z_2|^{2\alpha}}{\alpha^2}
\int_{\Omega_1} \frac{\dw}{|z_1-w|^{2\alpha}}.
\end{align}
Similarly,
\begin{align}\label{prop1_e3}
\int_{\Omega_2}
\left| \log \frac{|z_1-w|}{|z_2-w|} \right|^2 \dw
&\le
\frac{|z_1-z_2|^{2\alpha}}{\alpha^2}
\int_{\Omega_2} \frac{\dw}{|z_2-w|^{2\alpha}}.
\end{align}
Furthermore,
\begin{align}\label{prop1_e4}
\int_{\Om_i} \frac{\dw}{|z_i-w|^{2\alpha}}
&=\frac{1}{\pi}\int_{\Om_i} \frac{1}{|z_i-w|^{2\alpha}}\frac{dA(w)}{(1-|w|^2)^2}\leq \frac{1}{\pi (1-k^2)^2}\int_{\Om_i}\frac{dA(w)}{|z_i-w|^{2\alpha}} \nonumber \\
&
\le
\frac{1}{\pi (1-k^2)^2}
\int_{B(z_i,2)} \frac{dA(w)}{|z_i-w|^{2\alpha}}=\frac{1}{\pi (1-k^2)^2}2\pi \int_0^2 r^{1-2\alpha} \, dr\nonumber\\
&=\frac{2^{2-2\alpha}}{(1-k^2)^2 (1-\alpha)}=:C_\alpha,
\end{align}
where $B(z_i,2)$ is the Euclidean disk centered at $z_i$ with radius 2. Substituting \eqref{prop1_e2},\eqref{prop1_e3}, and \eqref{prop1_e4} in \eqref{prop1_e1}, we obtain
\begin{align}\label{prop1_e5}
\int_{\Om}
\left| \log \frac{|z_2-w|}{|z_1-w|} \right|^2 \dw
&\le
\frac{2C_\alpha}{\alpha^2}|z_1-z_2|^{2\alpha}.
\end{align}
Proceeding as in the previous case, we obtain
\begin{equation}
    \left| \log \frac{|1-\ol{z_1}w|}{|1-\ol{z_2}w|}\right|^2\leq \frac{1}{\alpha^2}\left(\frac{k}{1-k^2}\right)^{2\alpha} |z_1-z_2|^{2\alpha}.
\end{equation}
Therefore,
\begin{align}\label{prop1_e6}
    \int_\Om \left| \log \frac{|1-\ol{z_1}w|}{|1-\ol{z_2}w|}\right|^2 \dw &\leq  \frac{1}{\pi\alpha^2}\left(\frac{k}{1-k^2}\right)^{2\alpha} |z_1-z_2|^{2\alpha}\int_\Om\frac{dA(w)}{(1-|w|^2)^2}\nonumber\\
    &\leq \frac{|\Om|}{\pi\alpha^2}\frac{k^{2\alpha}}{(1-k^2)^{2+2\alpha}} |z_1-z_2|^{2\alpha}= C_\alpha' |z_1-z_2|^{2\alpha}.
\end{align}
Substituting \eqref{prop1_e5} and \eqref{prop1_e6} into \eqref{uniform_continuity_H}, we obtain
\begin{equation}\label{equicontinuous}
    |\Lom_h f(z_1)-\Lom_h f(z_2)|\leq C \|f\|_{L^2(\Om)} |z_1-z_2|^{\alpha}.
\end{equation}
Therefore, $\mathcal{L}_h f \in C^{\alpha}(\Omega)$, for $\alpha\in (0,1)$. In particular,
$\mathcal{L}_h f$ is H\"older continuous and hence uniformly continuous on
$\Omega$. Consequently, it admits a continuous extension to
$\overline{\Omega}$, and thus $\Lom_hf\in C(\overline{\Om})$.

Let \((f_n)_{n\in\mathbb{N}} \subset L^2(\Omega)\) be a bounded sequence with respect to $L^2$ norm. 
By \eqref{finalbound}, the sequence \((\mathcal{L}_h f_n)_{n\in\mathbb{N}}\) is uniformly bounded, and by the estimate in \eqref{equicontinuous}, it is equicontinuous. 
Hence, by the Arzela–Ascoli theorem, there exists a subsequence of \((\mathcal{L}_h f_n)_{n\in\mathbb{N}}\), which we relabel by the same index, that converges uniformly to some $g$ in \(C(\overline{\Omega})\).
Therefore, given \(\varepsilon>0\), there exists \(N\in\mathbb{N}\) such that
\begin{equation}
    \int_\Om (\Lom_h f_n(z)-g(z))^2 \dz \leq \frac{\epsilon^2}{\pi} \int_\Om \frac{dA(z)}{(1-|z|^2)^2}\leq   \frac{|\Om|}{\pi(1-k^2)^2}\epsilon^2, \quad\forall\, n\geq N,
\end{equation}
which implies the convergence of $(\Lom_h f_n)$ in $L^2(\Om)$. Therefore, the operator $\Lom_h$ is a compact operator on $L^2(\Om)$. Furthermore, $\Lom_h$ is a self-adjoint operator on $L^2(\Om)$, since:
\begin{equation*}
    \lb \Lom_h f,g\rb = \iint\limits_{\Om\;\Om} \log \frac{1}{[z,w]} f(z)g(w)\dz\dw= \lb \Lom_h g,f\rb.
\end{equation*}
\end{proof}
As a consequence of the above arguments, in particular \eqref{equicontinuous}, eigenfunctions corresponding to nonzero eigenvalues are continuous. We formalize this conclusion in the proposition below.
\begin{proposition}\label{prop_ef_regularity}
    Let $\Omega \subset \mathbb{D}$ be a bounded open set with respect to the hyperbolic metric on $\mathbb{D}$, and let $u$ be an eigenfunction corresponding to the eigenvalue $\tau$ of~\eqref{evproblem}. Then $\tau u \in C(\overline{\Omega})$.
\end{proposition}
\begin{proof}
    For any $f\in L^2(\Om)$, we have $\Lom_h f\in C(\ol{\Om})$, as follows from \eqref{equicontinuous}. In particular, $\Lom_h u =\tau u\in C(\overline{\Om}) $.
\end{proof}
Next, we show that the largest eigenvalue \(\tau_h\) is simple and  the corresponding eigenfunction can be chosen positive.
\begin{proposition}\label{prop-non-neg}
    Let $\Om\subset \disk$ be a bounded open subset of $\disk$ with respect to the hyperbolic metric on $\disk$. Then, an eigenfunction corresponding to $\tau_h$ does not vanish in $\Om$. Consequently, the eigenvalue $\tau_h$ is simple.
\end{proposition}
\begin{proof}
     Let $u_1$ be an eigenfunction corresponding to $\tau_h$ such that $\| u_1\|_{L^2(\Om)}=1$. From Proposition \ref{prop_ef_regularity} it follows that $u_1$ is continuous. Moreover,
    \begin{equation}\label{contradiction}
        E( u_1) = \sup \left\{ E(u):u\in L^2(\Omega),\|u\|_{L^2(\Omega)}=1\right\},
    \end{equation}
    where 
    \begin{equation}
        E(u):= \frac{1}{2}\iint\limits_{\Om\;\Om} \log \frac{1}{[z,w]} u(z)u(w)\dz\dw,\quad \text{for }u\in L^2(\Omega).
    \end{equation}
    Let
    \begin{equation*}
        \Omega^+=\{z\in\Omega: u_1(z)>0\},\,\,\text{  and  }\;
        \Omega^-=\{z\in\Omega: u_1(z)<0\}.
    \end{equation*}
    Suppose $u_1$ is sign changing, then 
    \begin{equation}\label{measnonzero}
        |\Omega^+|\neq 0 \text{ and }|\Omega^-|\neq 0.
    \end{equation}
   Moreover,
    \begin{equation}\label{signchange}
        |u_1(z)u_1(w)|>u_1(z)u_1(w),\quad \text{for all } z\in\Omega^+\, ,\, w\in\Omega^-.
    \end{equation}
    We have $[z,w]=|\phi_z(w)|<1,\text{ for any } z,w\in\Om$ and hence
    \begin{equation}\label{log is positive}
        \log \frac{1}{[z,w]} > 0,\quad \forall\, z,w\in\Om.
    \end{equation}
    Next, consider
    \begin{align*}
        E (|u_1|) &=\frac{1}{2}\iint\limits_{\Om^+\;\Om^+}\log\frac{1}{[z,w]} u_1(z)u_1(w)\dz\dw + \iint\limits_{\Om^+\;\Om^-}\log\frac{1}{[z,w]} |u_1(z) u_1(w)|\dz\dw\\
        &\quad\quad\quad\quad +\frac{1}{2} \iint\limits_{\Om^-\;\Om^-}\log\frac{1}{[z,w]} u_1(z)u_1(w)\dz\dw\\
        &> \frac{1}{2}\iint\limits_{\Om^+\;\Om^+}\log\frac{1}{[z,w]} u_1(z)u_1(w)\dz\dw +  \iint\limits_{\Om^+\;\Om^-}\log\frac{1}{[z,w]} u_1(z)u_1(w)\dz\dw\\
        &\quad\quad\quad\quad +\frac{1}{2}\iint\limits_{\Om^-\;\Om^-}\log\frac{1}{[z,w]}u_1(z)u_1(w)\dz\dw= E(u_1),
    \end{align*}
where the strict inequality follows from $\eqref{measnonzero},\eqref{signchange} $ and \eqref{log is positive}.
A contradiction to \eqref{contradiction}. Thus, $u_1$ can not change its sign in $\Omega$, and hence $u_1$ can be chosen to be non-negative.   
     Now suppose that $u_1(z_0)=0$ for some $z_0\in\Om$. Then,  
    \begin{equation}
        \tau_h u_1(z_0)=\frac{1}{2}\int_\Omega \log\frac{1}{[z_0,w]}u_1(w)\dw=0.
    \end{equation}
    Since the integrand does not change sign, from \eqref{log is positive}, we must have $u_1 \equiv 0$ in $\Omega$, a contradiction as $u_1$ is an eigenfunction. Thus, $u_1$ does not vanish in $\Om$.     
\end{proof}

In the \(\mathbb{R}^2\) case, the proof of the reverse Faber--Krahn inequality in \cite[Theorem~1.6]{anoopjiya2025} relies on a Riesz-type inequality; see \cite[Proposition~2.6]{anoopjiya2025}. Similarly, to prove Theorem~\ref{thm_hyperbolicfk}, we first establish a Riesz-type inequality. Consider the space
\begin{equation*}
    \mathcal{X}^h=\{f:f \text{ is measurable on }\disk \text{ and } \iint\limits_{\disk\;\disk} \left|\log \frac{1}{[z,w]}\right|\big| f(z)f(w) \big| \dz\dw<+\infty\}.
\end{equation*}
In the following proposition, we prove a Riesz-type inequality for functions in $\mathcal{X}^h$.
\begin{proposition}\label{prop_riezpol}
    Let $\mathcal{G}$ be a geodesic in $\mathbb{D}$ and let $\mathcal{H}$ be an associated polarizer. Let $f$ be a measurable function on $\disk$, such that $f,P_\H(f)\in\mathcal{X}^h$. Then,
    \begin{equation}\label{riez_type_inequality_pol}
        \iint\limits_{\disk\;\disk} \log \frac{1}{[z,w]} f(z)f(w) \dz\dw \leq  \iint\limits_{\disk\;\disk} \log \frac{1}{[z,w]} P_\H(f)(z)P_\H(f)(w) \dz\dw .
    \end{equation}
    In addition, the equality holds in \eqref{riez_type_inequality_pol} only if either $ P_\H(f)=f$ a.e. or $P_\H(f)=f\circ \sigma_\G $ a.e. in $\disk$.
\end{proposition}
\begin{proof}
    The proof follows the approach of \cite[Proposition 2.6]{anoopjiya2025}, with the necessary modifications for the Poincar\'e disk setting.
    \begin{align}
         I(f)&= \iint\limits_{\disk\;\disk} \log \frac{1}{[z,w]} f(z)f(w) \dz\dw\nonumber\\
         &=\iint\limits_{\H\;\H} \log \frac{1}{[z,w]} f(z)f(w) \dz\dw+ \iint\limits_{\H\;\widetilde{\H}} \log \frac{1}{[z,w]} f(z)f(w) \dz\dw\nonumber\\
         &\qquad\qquad +\iint\limits_{\widetilde{\H}\;\H} \log \frac{1}{[z,w]} f(z)f(w) \dz\dw+\iint\limits_{\widetilde{\H}\;\widetilde{\H}} \log \frac{1}{[z,w]} f(z)f(w) \dz\dw,
    \end{align}
    where $\widetilde{\H}=\sigma_\G(\H)$. By applying the change of variables from Proposition~\ref{prop_int_reflection}, we obtain
    \begin{equation}
        \iint\limits_{\H\;\widetilde{\H}} \log \frac{1}{[z,w]} f(z)f(w) \dz\dw=\iint\limits_{\H\;\H} \log \frac{1}{[\sigma_\G(z),w]} f(\sigma_\G(z))f(w) \dz\dw,
    \end{equation}
    and 
    \begin{equation}
        \iint\limits_{\widetilde{\H}\;\widetilde{\H}} \log \frac{1}{[z,w]} f(z)f(w) \dz\dw=\iint\limits_{\H\;\H} \log \frac{1}{[\sigma_\G(z),\sigma_\G(w)]} f(\sigma_\G(z))f(\sigma_\G(w)) \dz\dw.
    \end{equation}
\noindent Denote $\widetilde{z}=\sigma_\G(z)$ and $K(z,w)=\log \frac{1}{[z,w]}$. Thus
\begin{align}
    I(f)&=\iint\limits_{\H\;\H} \big[ K(z,w)(f(z)f(w)+f(\widetilde{z})f(\widetilde{w}))+K(z,\widetilde{w})(f(z)f(\widetilde{w})+f(\widetilde{z})f(w)) \big] \dz\dw.
\end{align}
 Let 
 \begin{equation*}
     \mathcal{S}_{f}(z,w):=K(z,w)(f(z)f(w)+f(\widetilde{z})f(\widetilde{w}))+K(z,\widetilde{w})(f(z)f(\widetilde{w})+f(\widetilde{z})f(w)).
 \end{equation*}
  Then,
  \begin{equation}\label{I_f}
     I(f)= \iint\limits_{A\;A} \mathcal{S}_{f}(z,w) \dz\dw + 2 \iint\limits_{B\;A} \mathcal{S}_{f}(z,w) \dz\dw +\iint\limits_{B\;B} \mathcal{S}_{f}(z,w) \dz\dw,
 \end{equation}
 where  
 \begin{equation*}
     A=\{z\in \H: f(z)\geq f(\widetilde{z})\}\text{ and }B=\{z\in \H: f(z)< f(\widetilde{z})\} .
 \end{equation*}
 Notice that,
 \begin{equation}\label{x_in A_x_in B def}
\begin{aligned}
  P_\H(f)(z)=f(z) &\quad\text{and}\quad P_\H(f)(\widetilde{z})=f(\widetilde{z}),\quad\forall\, z\in A,\\
     P_\H(f)(z)=f(\widetilde{z}) &\quad\text{and}\quad  P_\H(f)(\widetilde{z})=f(z),\quad\forall\, z\in B.
\end{aligned}
\end{equation}
It is easy to verify that $$\mathcal{S}_{f}(z,w)=\mathcal{S}_{P_\H(f)}(z,w),\quad \forall\, (z,w)\in (A\times A) \bigcup (B\times B).$$ Consequently,
\begin{equation}\label{A_A}
    \iint\limits_{A\;A} \mathcal{S}_{f}(z,w) \dz\dw= \iint\limits_{A\;A} \mathcal{S}_{P_\H(f)}(z,w) \dz\dw,
\end{equation}
and
\begin{equation}\label{B_B}
    \iint\limits_{B\;B} \mathcal{S}_{f}(z,w) \dz\dw= \iint\limits_{B\;B} \mathcal{S}_{P_\H(f)}(z,w) \dz\dw.
\end{equation}

\noindent On the other hand, for $(z,w)\in A\times B$, 
    \begin{equation}\label{S_PHf-S_f}
        \mathcal{S}_{P_\H(f)}(z,w)-\mathcal{S}_{f}(z,w)=(f(z)-f(\widetilde{z}))(f(\widetilde{w})-f(w))(K(z,w)-K(z,\widetilde{w})).
    \end{equation}
    \noindent From Proposition \ref{prop_reflection}, it follows that  $[z,w]< [z,\widetilde{w}]$, and hence  $K(z,w)>K(z,\widetilde{w})$. Thus,
    \begin{equation*}
        \mathcal{S}_{P_\H(f)}(z,w)-\mathcal{S}_{f}(z,w)\geq 0.
    \end{equation*}
    Therefore,
    \begin{equation}\label{B_A}
    \iint\limits_{B\;A} \mathcal{S}_{f}(z,w) \dz\dw\leq \iint\limits_{B\;A} \mathcal{S}_{P_\H(f)}(z,w) \dz\dw.
\end{equation}

\noindent Now, by combining \eqref{A_A}, \eqref{B_B}and \eqref{B_A}, we conclude  $I(f)\leq I(P_\H(f))$. Next, we assume that  $I(f)=I(P_\H(f))$. Thus,
\begin{equation*}
    \iint\limits_{B\;A} \mathcal{S}_{f}(z,w)\dz\dw=\iint\limits_{B\;A} \mathcal{S}_{P_\H(f)}(z,w)\dz\dw.
\end{equation*}
Therefore, $$\mathcal{S}_{P_\H(f)}(z,w)=\mathcal{S}_{f}(z,w) \text{ a.e on } A\times B.$$ Now, let $A_1=\{z\in A: f(z)>f(\widetilde{z})\}.$ 
Since $K(z,w)>K(z,\widetilde{w})$,  from \eqref{S_PHf-S_f} we conclude that $\mathcal{S}_{P_\H(f)}(z,w)>\mathcal{S}_{f}(z,w)$, for $(z,w)\in A_1\times B$. Therefore $|A_1\times B|=0$ and hence $|A_1|=0$ or $|B|=0$. Notice that, 
\begin{enumerate}[(i)]
    \item if $|B|=0$, then $f(z)\geq f(\widetilde{z})$  a.e. in $H$, and hence $P_\H(f)=f$ a.e. in $\disk,$
     \item if $|A_1|=0$, then $f(z)\leq f(\widetilde{z})$ for a.e. in $H$, and hence $P_\H(f)=f\circ\sigma_\G$  a.e. in $\disk$.  
\end{enumerate}
 This concludes the proof.
\end{proof}
Using this Riesz-type inequality, we now prove the reverse Faber-Krahn inequality stated in the Introduction.
\\
\noindent\textbf{Proof of Theorem \ref{thm_hyperbolicfk}:}
\par 
Let $u_1$ be the eigenfunction corresponding to $\tau_h(\Om)$, normalized by $\|u_1\|_{L^2(\Omega)}=1$ and satisfying $u_1>0$ in $\Om$ (Proposition \ref{prop-non-neg}). By Proposition~\ref{prop_pol_properties}, we have   $\|P_\H(u_1)\|_{L^2(P_\H(\Omega))}=\|u_1\|_{L^2(\Omega)}=1$ and $P_\H(u_1)=0$ on $P_\H(\Om)^c$. Using Proposition~\ref{prop_riezpol}, we obtain
\begin{align}\label{logpolid}
        \tau_h(\Omega)&=\frac{1}{2}\iint\limits_{\Om\;\Om}\log\frac{1}{[z,w]}u_1(z)u_1(w) \dz\dw\nonumber\\
        &\leq \frac{1}{2} \int\limits_{P_\H(\Om)} \int\limits_{P_\H(\Om)} \log\frac{1}{[z,w]}P_\H(u_1)(z)P_\H(u_1)(w)\dz\dw\leq \tau_h(P_\H(\Omega)).\nonumber
    \end{align} 
\noindent Next, assume that $\tau_h(\Omega)=\tau_h(P_\H(\Omega))$. Then, from the above inequality, we obtain
    \begin{equation*}
        \frac{1}{2}\iint\limits_{\Om\;\Om} \log\frac{1}{[z,w]}u_1(z)u_1(w)\dz\dw\\
        = \frac{1}{2}\int\limits_{P_\H(\Om)} \int\limits_{P_\H(\Om)}\log\frac{1}{[z,w]}P_\H(u_1)(z)P_\H(u_1)(w) \dz\dw.
    \end{equation*}
     Now, as a consequence of Proposition \ref{prop_riezpol}, we have  either
     \begin{equation}\label{two_possible_cases}
         P_\H(u_1)=u_1\text{ a.e.  in }\disk \text{ or } P_\H(u_1)=u_1\circ \sigma_\G\text{ a.e.  in }\disk.
     \end{equation}
    One can easily verify that,
    \begin{equation*}
        P_\H(u_1)\neq u_1\text{ on }P_\H(\Om)\,\triangle\, \Om,
    \end{equation*}
    and 
    \begin{equation*}
        P_\H(u_1)\neq u_1\circ\sigma_\G\text{ on }P_\H(\Om)\,\triangle\, \sigma_\G(\Om).
    \end{equation*}
     By \eqref{two_possible_cases}, we must have $|P_\H(\Om)\,\triangle\, \Om|=0$ or $|P_\H(\Om)\,\triangle\, \sigma_\G(\Om)|=0$. 
    
\qed
\hfill\\
\par Next, we derive an explicit representation for the eigenfunctions of \eqref{evproblem}, analogous to the representation for \eqref{oldevproblem} proved by Troutman in \cite[Lemma, p.~67]{troutman1967}.
\\
\hfill\\
\noindent\textbf{Proof of Theorem \ref{thm_efreptheorem1}:}
\\
    Let $u$ be an eigenfunction corresponding to a nonzero eigenvalue $\tau$. Then, by \eqref{ef_on_D}
    \begin{equation}
        -2 \tau u(z)=\int_\Om \log[z,w]u(w)\dw, \quad \forall\, z\in \disk.
    \end{equation}
    Fix $z\in\disk$ and $0<r<1-|z|$. Replacing $z$ by $\phi_z(re^{i\theta})$,
    \begin{align*}
        -2 \tau u(\phi_z(re^{i\theta}))&=\int_\Om \log[\phi_z(re^{i\theta}),w]u(w)\dw\\
        &= \int_\Om \log [re^{i\theta},\phi_z(w)]u(w)\dw,\qquad(\text{Proposition }\eqref{prop_mobius_preseve_dist})\\
        &= \int_\Om \log |\phi_{re^{i\theta}}(\phi_z(w))|u(w)\dw\\
        &=\int_\Om \log \left| \frac{\phi_z(w)-re^{i\theta}}{1-re^{i\theta }\ol{\phi_z(w)}} \right|u(w)\dw
    \end{align*}  
    Integrating both sides with respect to  $\theta$, and by applying Fubini's theorem, we obtain
    \begin{align}\label{e3_h}
        -2 \tau \int_0^{2\pi} u(\phi_z(re^{i\theta}))d\theta&= \int_\Om \int_0^{2\pi} \log \left| \phi_z(w)-re^{i\theta} \right|d\theta\, u(w)\dw\nonumber\\
        &\quad\quad - \int_\Om\int_0^{2\pi} \log \left| 1-re^{i\theta }\ol{\phi_z(w)} \right|d\theta \, u(w)\dw
    \end{align}
    Recall,
    \begin{equation*}
        \int_0^{2\pi} \log (|1-ae^{i\theta}|)d\theta =\begin{cases}
            0, & |a|\leq 1,\\
            2\pi \log |a|, &|a|>1.
        \end{cases}
    \end{equation*}
    Since $0<r<1$ and $\phi_z:\mathbb{D}\ra \mathbb{D}$, it follows that $$\int_0^{2\pi} \log \left| 1-re^{i\theta }\ol{\phi_z(w)} \right|d\theta=0,$$ 
    and hence the second integral vanishes. Next, consider
    \begin{align*}
        \int_0^{2\pi} \log \left| \phi_z(w)-re^{i\theta} \right|d\theta&=\int_0^{2\pi} \log |\phi_z(w)|d\theta+\int_0^{2\pi}\log \left| 1-\frac{re^{i\theta}}{\phi_z(w)} \right|d\theta\\
        &=\begin{cases}
            2\pi \log \left| \phi_z(w)\right|, &w\in \disk\setminus \Delta_r(z)\\
            2\pi \log \left| \phi_z(w)\right|-2\pi \log \left|\frac{\phi_z(w)}{r}\right|, & w\in \Delta_r(z)
        \end{cases}
    \end{align*}
    Substituting in \eqref{e3_h} gives
    \begin{align*}
        -2 \tau \int_0^{2\pi} u(\phi_z(re^{i\theta}))d\theta&=2\pi \int_{\Om} \log \left| \phi_z(w)\right|u(w)\dw - 2\pi\int_{\Omega\cap\Delta_r(z)} \log \left|\frac{\phi_z(w)}{r}\right| u(w)\dw,        
    \end{align*}
    and thus dividing by $-4\pi$, we get
    \begin{equation}
        \frac{\tau}{2\pi} \int_0^{2\pi} u(\phi_z(re^{i\theta}))d\theta =  \tau u(z)+\frac{1}{2}\int_{\Om\cap \Delta_r(z)} \log \frac{[z,w]}{r} u(w)\dw
    \end{equation}
    This completes the proof. In the case $\tau = 0$, the same argument applies; however, the corresponding eigenfunction cannot be extended to $\mathbb{D}$ as given in \eqref{ef_on_D}, and hence the representation is valid only inside $\Omega$.   
\qed

\noindent From this representation, we next establish that zero is not an eigenvalue of \eqref{evproblem}.
\\
\hfill\\
\noindent\textbf{Proof of Theorem \ref{cor_zero}:}
\\
Suppose $0$ is an eigenvalue and let $u$ be a corresponding eigenfunction.
Fix $z_0 \in \Omega$ and choose $r>0$ such that $\Delta_r(z_0)\subset \Omega$. Then from the representation formula \eqref{efrep1_hyper},
\begin{equation}\label{eq:rep-zero}
    \int_{\Delta_s(z_0)} \log\frac{[z_0,w]}{s}\, u(w)\, \dw = 0,
\qquad \forall\, 0<s\leq r .
\end{equation}
Define
\[
F_{z_0}(s):=\int_{\Delta_s(z_0)} \log\frac{[z_0,w]}{s}\, u(w)\, \dw,
\qquad 0< s< r .
\]
Let $0<s<t<r$. Then
\begin{align*}
F_{z_0}(t)-F_{z_0}(s)
&=
\int_{\Delta_t(z_0)} \log \frac{[z_0,w]}{t}\, u(w)\, \dw
-
\int_{\Delta_s(z_0)} \log \frac{[z_0,w]}{s}\, u(w)\, \dw \\
&=
\int_{\Delta_t(z_0)\setminus \Delta_s(z_0)}
\log \frac{[z_0,w]}{t}\, u(w)\, \dw
+
\int_{\Delta_s(z_0)} \log \frac{s}{t}\, u(w)\, \dw .
\end{align*}
Dividing by $t-s$, we obtain
\begin{align*}
\frac{F_{z_0}(t)-F_{z_0}(s)}{t-s}
&=
\frac{1}{t-s}
\int_{\Delta_t(z_0)\setminus \Delta_s(z_0)}
\log \frac{[z_0,w]}{t}\, u(w)\, \dw \\
&\quad
+
\frac{\log s - \log t}{t-s}
\int_{\Delta_s(z_0)} u(w)\, \dw .
\end{align*}
By using the Cauchy-Schwarz inequality,
\begin{align*}
\left|
\frac{1}{t-s}
\int_{\Delta_t(z_0)\setminus \Delta_s(z_0)}
\log \frac{[z_0,w]}{t}\, u(w)\, \dw
\right|
&\le
\frac{1}{t-s}
\left(
\int_{\Delta_t(z_0)\setminus \Delta_s(z_0)}
\left|\log \frac{[z_0,w]}{t}\right|^2 \dw
\right)^{1/2} \\
&\quad \times
\left(
\int_{\Delta_t(z_0)\setminus \Delta_s(z_0)}
|u(w)|^2 \dw
\right)^{1/2}\\
&\leq \frac{\|\Om\|\|u\|_{L^2(\Om)}}{t-s}(\log t - \log s)^2 ,
\end{align*}
where $\|\Om\|$ denotes the hyperbolic area measure of $\Om$. Moreover,
\[
\lim_{s \to t}
\frac{(\log t - \log s)^2}{t-s}
=
\lim_{s \to t}
\frac{2(\log t - \log s)}{s}
=0 .
\]
Hence,
\[
\lim_{s \to t}
\frac{F_{z_0}(t)-F_{z_0}(s)}{t-s}
=
-\frac{1}{t}
\int_{\Delta_t(z_0)} u(w)\, \dw .
\]
That is,
\[
F_{z_0}'(t)
=
-\frac{1}{t}
\int_{\Delta_t(z_0)} u(w)\, \dw ,
\qquad 0<t<r .
\]
However, $F_{z_0}(t)=0$ for all $0< t< r$. Consequently 
\[
F_{z_0}'(t)=0,
\quad \text{and}\quad 
\int_{\Delta_t(z_0)} u(w)\, \dw = 0
\quad 0<t<r .
\]
Recall that since $\Omega$ is bounded, there exists a constant $k<1$ such that $|z|\le k, \text{ for all } z\in\Omega,$ as in the proof of Proposition~\ref{prop_L_h_compact_sa}.  
Moreover, for $0<t<r$, the Lebesgue measure of $\Delta_t(z)$, denoted by  $|\Delta_t(z)|=\pi\,\frac{(1-|z|^2)^2\,t^2}{\bigl(1-|z|^2 t^2\bigr)^2},$ see \cite[Theorem~2.2.2]{Stoll_2016}. Since each hyperbolic disk $\Delta_t(z)\subset B(z,2t)$ and $$|\Delta_t(z)|\geq \frac{\pi t^2(1-k^2)^2}{4}\geq C |B(z,2t)|,$$
 by the Lebesgue differentiation theorem (see \cite[p.98]{folland1999}, \cite[p.106 \& p.108]{SteinShakarchi}), we have
\[
\frac{1}{\pi}\frac{u(z)}{(1-|z|^2)^2}=\lim_{t\to 0}\frac{1}{|\Delta_t(z)|}\int_{\Delta_t(z))}u(w)\,\frac{dA(w)}{\pi(1-|w|^2)^2}=0
\quad \text{for a.e. } z\in\Omega .
\]
By \eqref{eq:rep-zero}, the above integral vanishes for all sufficiently
small $t$, hence $u=0$ a.e. in $\Omega$, contradicting the assumption that $u$ is an eigenfunction. Thus, $0$ is not an eigenvalue. 
\qed
\vspace{2mm}
\\The following remark is required to prove Theorem \ref{thm_hyper_transfinite}.
\begin{remark}\label{rem_boundary} 
        Let $\Omega \subset \mathbb{D}$ be an open set, bounded with respect to the hyperbolic metric $d_h$. Let $u$ be an eigenfunction of $\Lom_h$ corresponding to the eigenvalue $\tau$. Notice that 
        \[
        -2\,\mathcal{L}_h u(z)
        = \int_{\Omega} \log [z,w]\, u(w)\, \dw
        = \int_{\Omega} \log \frac{[z,w]}{|z|}\, u(w)\, \dw
        + \log |z| \int_{\Omega} u(w)\, \dw .
        \]
        Moreover,
        \[
        \frac{[z,w]}{|z|}
        = \frac{|z-w|}{|z|\;|1-\overline{z}w|}
        = \frac{|z-w|}{|z - |z|^{2} w|}.
        \]
        As \(|z| \to 1\), we have \(\frac{[z,w]}{|z|} \to 1\). By the dominated convergence theorem,
        \[
        \int_{\Omega} \log \frac{[z,w]}{|z|}\, u(w)\, \dw \rightarrow 0,
        \]
        and the second term also converges to zero. Therefore, \(\mathcal{L}_hu(z) \to 0\) as $|z|\rightarrow 1$. Consequently, for an eigenfunction \(u\) corresponding to the eigenvalue \(\tau\),
        \[
        \tau\,u(z) \to 0 \quad \text{as } |z| \to 1.
        \]
        By Corollory \ref{cor_zero}, we have $\tau\neq 0$, and consequently 
        \[
        u(z) \to 0 \quad \text{as } |z| \to 1.
        \]
\end{remark}
In \cite[Theorem 3]{troutman1967}, it is shown that the logarithmic potential operator $\mathcal{L}$ admits a negative eigenvalue if and only if the transfinite diameter of $\Omega \subset \mathbb{R}^2$ exceeds 1. Consequently, $\mathcal{L}$ is a positive operator if and only if the transfinite diameter of $\Omega$ is less than or equal to 1. In the present setting, the hyperbolic logarithmic operator $\mathcal{L}_h$ is always positive, in contrast with the Euclidean case. We next prove Theorem~\ref{thm_hyper_transfinite}.
\\
\hfill\\
\noindent\textbf{Proof of Theorem \ref{thm_hyper_transfinite}:}
\par Suppose, to the contrary, that $\mathcal{L}_h$ is not a positive operator. Then there exists $f \in L^2(\Omega)$ such that $\lb \Lom_h f, f\rb<0$. Consequently, 
\begin{equation}\label{weightedcharmu1_neg}
\widetilde{\tau_h}(\Omega)
:= \inf \left\{
\frac{1}{2}
\iint\limits_{\Om\;\Om}
\log \frac{1}{[z,w]}  u(z) u(w) \dz \dw
: u \in L^2(\Omega),\ \int_\Omega u^2 = 1
\right\}<0.
\end{equation}
Then, it follows that $\widetilde{\tau_h}(\Omega)$ is an eigenvalue of $\mathcal{L}_h$\cite[Proposition 8.5.2]{Keshfunctional2022}. Next, we show that any eigenfunction $u$ corresponding to $\widetilde{\tau_h}(\Omega)$, if it exists, must be identically zero. Let $u$ be an eigenfunction corresponding to $\widetilde{\tau_h}(\Omega)$.
Suppose that $u<0$ at some point. By Remark~\ref{rem_boundary}, $u$ must
attain a strictly negative minimum in $\mathbb{D}$, say at $z_{0}$.
Consider a sufficiently small neighbourhood
$\Delta_{r}(z_{0}) \subset \mathbb{D}$ such that, for all
$w \in \Delta_{r}(z_{0})$,
\begin{equation}
u(z_{0}) \leq u(w) < 0.
\end{equation}
In particular, for
\begin{equation}
w = \phi_{z_{0}}(s e^{i\theta}), \qquad s<r,
\end{equation}
we have
\begin{equation}
[z_{0},w]
= [\phi_{z_{0}}(z_{0}),\phi_{z_{0}}(w)]
= [0,s e^{i\theta}]
= |s e^{i\theta}| = s < r,
\end{equation}
and hence $w\in \Delta_{r}(z_{0})$. Therefore,
\begin{equation}\label{e1_h}
u(z_{0})
\leq \frac{1}{2\pi} \int_{0}^{2\pi}
u\!\left(\phi_{z_{0}}(s e^{i\theta})\right)\, d\theta .
\end{equation}
From \eqref{efrep1_hyper}, we obtain
\begin{equation}
u(z_{0})
= \frac{1}{2\pi} \int_{0}^{2\pi}
u\!\left(\phi_{z_{0}}(s e^{i\theta})\right)\, d\theta
- \frac{1}{2\widetilde{\tau_h}(\Omega)}
\int_{\Omega \cap \Delta_{s}(z_{0})}
\log\frac{[z_{0},w]}{s} u(w)\, \dw .
\end{equation}
Since $\widetilde{\tau_h}(\Omega)<0$, and
\begin{equation}
    \log\frac{[z_{0},w]}{s}<0, \text{ and } u(w)<0 \text{ for all } w \in \Delta_{s}(z_{0}), 
\end{equation}
it follows that
\begin{equation}\label{e2_h}
u(z_{0})
\geq \frac{1}{2\pi} \int_{0}^{2\pi}
u\!\left(\phi_{z_{0}}(s e^{i\theta})\right)\, d\theta .
\end{equation}
Combining \eqref{e1_h} and \eqref{e2_h}, we conclude that
\begin{equation}\label{lastthm_last eqn}
u(z_{0})
= \frac{1}{2\pi} \int_{0}^{2\pi}
u\!\left(\phi_{z_{0}}(s e^{i\theta})\right)\, d\theta .
\end{equation}
Substituting this back into \eqref{efrep1_hyper}, we further obtain
\begin{equation}
\int_{\Omega \cap \Delta_{s}(z_{0})}
\log\frac{[z_{0},w]}{s} u(w)\, \dw = 0,
\qquad s<r.
\end{equation}
This leads to a contradiction, since
$\log\frac{[z_{0},w]}{s}<0$ and $u(w)<0$ in
$\Delta_{s}(z_{0})$. An analogous argument yields a contradiction when $u>0$ at some point in $\Om$; by Remark~\ref{rem_boundary},
$u$ attains a strictly positive maximum in $\mathbb{D}$, say at $z_{1}$.
Choose $\tilde{r}>0$ sufficiently small so that $\triangle_{\tilde{r}}(z_{1}) \subset \mathbb{D}$
and
\begin{equation}
u(z_{1}) \geq u(w) > 0 \qquad \text{for all } w \in \triangle_{\tilde{r}}(z_{1}).
\end{equation}
In this case, the inequalities in \eqref{e1_h} and \eqref{e2_h} are reversed. Neverthless, the arguments from \eqref{lastthm_last eqn} onwards remain the same and yield a similar contradiction.
Thus, $u\equiv 0$ and hence $\widetilde{\tau_h}(\Omega)$ cannot be an eigenvalue of $\mathcal{L}_h$. Therefore, $\mathcal{L}_{h}$ must be a positive operator on $L^2(\Om)$.

\qed
\vspace{5mm}
\noindent{\textbf{Acknowledgment}}
\par The authors express their sincere gratitude to Dr. Anoop T. V. for helpful discussions related to this manuscript.
S. Verma acknowledges the project grant provided by CSIR-ASPIRE sanction order no. 25WS(011)/2023-24/EMR-II/ASPIRE.
\bibliographystyle{abbrvurl}
 \bibliography{Reference}

\end{document}